\documentclass[11pt]{article}
\usepackage{arxiv}
\usepackage[utf8]{inputenc}
\usepackage{graphicx}%
\usepackage{multirow}%
\usepackage{amsmath,amssymb,amsfonts}%
\usepackage{amsthm}%
\usepackage{mathrsfs}%
\usepackage[title]{appendix}%
\usepackage{xcolor}%
\usepackage{textcomp}%
\usepackage{manyfoot}%
\usepackage{booktabs}%
\usepackage{longtable}%
\usepackage{hyperref}
\hypersetup{
    colorlinks=true,
    linkcolor=blue,
    filecolor=magenta,      
    urlcolor=blue,
    citecolor=blue
}

\usepackage[sort&compress]{natbib}
 \bibpunct[, ]{[}{]}{,}{n}{}{,}%
\newtheorem{theorem}{Theorem}%
\newtheorem{proposition}[theorem]{Proposition}%
\newtheorem{lemma}[theorem]{Lemma}%

\newtheorem{remark}{Remark}%

\newtheorem{definition}{Definition}%
\newtheorem{assumption}{Assumption}%

\newcommand{\R}{\mathbb{R}}
\newcommand{\E}{\mathbb{E}}
\newcommand{\Pprob}{\mathbb{P}}
\newcommand{\Filt}{\mathscr{F}}
\newcommand{\Gfilt}{\mathcal{G}}
\newcommand{\Tr}{\operatorname{Tr}}
\newcommand{\diff}{\mathrm{d}}
\newcommand{\Proj}{\operatorname{Proj}}
\newcommand{\Acal}{\mathcal{A}}
\newcommand{\Lcal}{\mathcal{L}}
\newcommand{\Ucal}{\mathcal{U}}
\newcommand{\Qcal}{\mathcal{Q}}
\newcommand{\Pcal}{\mathcal{P}}
\newcommand{\Ccal}{\mathcal{C}}
\newcommand{\Vcal}{\mathcal{V}}
\newcommand{\Sym}{\mathbb{S}}
\newcommand{\Hsing}{\mathcal{H}}
\newcommand{\Lleb}{\mathcal{L}}

\begin{document}

\title{Continuous-Time Information Design for Hurricane Evacuation: Disclosure, Congestion, and Optimal Phasing under Model Uncertainty}
\author{Furkan Sezer\thanks{Texas A\&M University (furkan.sezer@tamu.edu, furkanszr@yahoo.com).}}
\date{\today}

\maketitle

\begin{abstract}
We study continuous-time information design for emergency evacuation, where an Emergency Management Agency (the Stackelberg leader) steers strategic evacuation zones via two levers: \textit{public advisory precision} (information design) and a \textit{tiered release schedule}. The latent storm is a jump-diffusion process with publicly observed rapid-intensification epochs tracked by an exact finite-dimensional belief filter. Zones play a capacity-constrained congestion game on shared corridors with belief-weighted hazard exposure. The running cost couples beliefs to a convex congestion externality, making disclosure double-edged: sharper information reduces false-alarm departures but synchronizes genuine ones, and convex congestion penalizes that synchronization. We prove that: (i) the followers' game admits a potential reduction to a convex control problem; (ii) the leader's distributionally robust relative-entropy problem is characterized by an Isaacs equation whose value is the unique viscosity solution, with verification valid for non-smooth bang-bang feedback; and (iii) without transfers, the leader's first-order condition retains an equilibrium-response term, positioning optimal information design as a second-best congestion toll. Structurally, we show that a staggered evacuation order dominates simultaneous advisories; phased evacuation emerges endogenously as optimal information design. Furthermore, public-signal precision is sign-ambiguous due to an informational Braess effect, where vague advisories are optimal unless complemented by a staggered order. Calibrated to Hurricane Rita using NHC archives, TxDOT capacities, and HRRC surveys, the model reproduces the observed gridlock along the Interstate 45 (I-45) evacuation corridor in Texas. The optimal policy removes essentially all in-transit congestion exposure, reducing social cost by $\mathbf{89\%}$, while staggered disclosure alone yields a $\mathbf{70\%}$ reduction.
\end{abstract}

\keywords{information design, differential games, Stackelberg equilibrium, viscosity solutions, congestion games, hurricane evacuation}

\vspace{0.3cm}
\noindent {\small \textbf{Mathematics Subject Classification (2020):} 91A65, 93E20, 49L25, 91A23, 90C47, 49N90, 35Q93}

\maketitle

\section{Introduction}
Climate-induced extreme weather events present some of the most complex,
high-stakes operational challenges for public asset management and regional
emergency response. When intense meteorological shocks, Hurricanes Katrina, Rita,
Harvey, Irma, or Ian, threaten dense coastal corridors, public safety hinges on
clearing large populations before landfall. Two opposite failures recur. In
Hurricane Rita (2005), an alarming and undifferentiated public message triggered a
near-simultaneous mass departure that saturated the Houston evacuation corridors;
the resulting gridlock, not the storm, became the dominant source of casualties.
In Hurricane Harvey (2017), the absence of a mandatory order left a major
metropolitan population in place during catastrophic flooding. The first is an
\emph{over-disclosure} (synchronisation) failure; the second an
\emph{under-disclosure} (inertia) failure. A theory of evacuation information
design must explain both, and must locate the optimal advisory \emph{between} them.

Historically, emergency management has relied on coarse, deterministic, and
decentralised command. Local municipalities issue orders independently, under
fragmented information and timelines, while the physical infrastructure exhibits
hard, non-linear capacity limits: when multiple dense zones flood shared arteries,
the network transitions abruptly from free flow to gridlock cascade, trapping
evacuees in exposed clusters. Public advisories from the National Hurricane Center
(NHC) and FEMA are powerful coordination tools, but risk communication has
traditionally been treated as passive meteorological transparency, ignoring the
behavioural externalities it triggers on the ground.

We formalise evacuation as a continuous-time stochastic Stackelberg game. A central
agency commits to a public-signalling channel and a tiered release schedule; a
network of zones plays a capacity-constrained congestion game whose hazard exposure
is governed by their beliefs about the storm. The single structural departure from
recent joint information--mechanism design models for power systems~\cite{sezer2026power} is that the followers' cost couples beliefs to a
\emph{convex shared-corridor congestion} term. As we show, this makes disclosure
double-edged and produces qualitative results that monotone-benefit information
design cannot.

\subsection{Contributions and main results.}
\begin{enumerate}\itemsep2pt
\item \textbf{Well-posed belief dynamics.} The jump-diffusion storm admits a
finite-dimensional belief filter that is \emph{exact} between publicly observed
rapid-intensification epochs, with a covariance reset at each
(Lemma~\ref{lem:filter}).
\item \textbf{Congestion subgame and externality.} Zones play a potential
congestion game on shared corridors; the Wardrop equilibrium over-loads corridors
relative to the social optimum, the price-of-anarchy wedge
(Remark~\ref{rem:potential}). The gridlock-cascade threshold
$\rho(\mathbf\Gamma^{-1}H)<1$ governs the conditioning of that equilibrium.
\item \textbf{Robust master problem.} The leader's distributionally robust
Stackelberg problem over $(\alpha,\theta)$ is characterised by an Isaacs equation (Section \ref{sub:master}).
\item \textbf{Potential reduction.} The followers' game collapses to a single
strictly convex control problem with the \emph{Beckmann} running cost, not social
welfare; the gap is the externality (Lemma~\ref{lem:potential-reduction}).
\item \textbf{Viscosity characterisation.} The value is the unique viscosity
solution of the Isaacs equation, treated piecewise across activation epochs;
verification holds without smoothness and semiconcavity yields a Lebesgue-null
switching set and a well-posed Filippov closed loop
(Theorems~\ref{thm:existence-evac}--\ref{thm:viscverify-evac};
Proposition~\ref{prop:semiconcave-evac}).
\item \textbf{Information as a second-best toll.} With no incentive-aligning
transfer, the leader's stationarity retains an MPEC equilibrium-response term equal
to the marginal externality times the equilibrium sensitivity; it vanishes iff no
corridor binds (Proposition~\ref{prop:mpec}).
\item \textbf{The value of precision is sign-ambiguous (informational Braess).}
Sharpening a single public advisory synchronises departures and can \emph{raise}
social cost, so a deliberately vague signal is optimal in the cascade regime, and
precision becomes valuable only once orders are staggered, with which it is
complementary; in the scalar symmetric case this sign-reversal is explicit, with an
interior/corner disclosure threshold (Proposition~\ref{prop:interior-explicit}).
\item \textbf{Tiered phasing as optimal information.} A single simultaneous advisory
is dominated: the optimal information structure is a publicly announced, staggered
evacuation order, so phased evacuation emerges endogenously as optimal information
design; the scalar symmetric case gives the explicit two-tier threshold and
continuous-tier fill-to-capacity limit in closed form (Proposition~\ref{prop:twotier}).
\item \textbf{Calibration to Hurricane Rita.} Calibrated to the 2005 event, two experiments quantify the two levers. When the regulator meters egress directly (Experiment 1), the realised synchronised order reproduces the documented I-45 gridlock (1–2 mph), while the solved optimum spreads egress to hold the corridor at capacity, cutting social cost by $89\%$ and the over-capacity exposure from 46.6 to 0 corridor-hours. Simpler coarse two-tier stagger approach brings $77\%$ cost reduction and reducing exposure to 5 corridor-hours as well. When the regulator can only disclose (Experiment 2), a staggered-disclosure design lowers social cost by $70\%$ relative to a single simultaneous advisory, and signal precision is valuable only in combination with that staggering (Section \ref{sec:numerics}).
\end{enumerate}

\noindent\fbox{\parbox{0.97\linewidth}{\small\textbf{Problem at a glance.} The
controlled state is $(\hat X_t,\Pi_t,Y_t)$: the belief mean and error covariance of
the jump-diffusion storm (a finite-dimensional sufficient statistic between
observed intensification jumps) and the zones' at-risk backlogs $Y_t$. The leader's
controls are the disclosure gain $\alpha_t$ (advisory precision) and a tiered
activation schedule $\theta$ (who is told to go, and when); the leader does
\emph{not} dispatch traffic. Zones choose egress rates, endogenous through a
lower-level congestion (Wardrop) equilibrium on shared corridors. The leader
minimises the worst-case (relative-entropy) social cost, belief-weighted hazard
exposure plus convex corridor congestion, anticipating the zones' equilibrium
response.}}

\section{Related Work and Positioning}\label{sec:related}
Our work sits at the confluence of information design, information-design-for-congestion,
evacuation operations research, robust control, and nonlinear filtering.

\subsection{Information design and Bayesian persuasion.}
The disclosure layer descends from Bayesian persuasion~\cite{kamenica2011bayesian}
and dynamic/continuous-time information design~\cite{ely2017beeps,koessler2023informed},
with multi-receiver public-signal structures formalised as
correlated-equilibrium recommendations~\cite{bergemann2019information}. We use the
public, zone-addressed recommendation as the primitive (orders are broadcast and
commonly observed), and add a jump-diffusion latent state and a physical congestion
externality absent from that literature.

Closest in spirit is the line on information provision in routing and congestion:
optimal traffic-information disclosure and obfuscation, and the
\emph{informational Braess paradox} in which more information can raise
congestion~\cite{das2017reducing,acemoglu2018informational,tavafoghi2017informational}.
That literature is largely static and network-routing-centric; we lift it to a
continuous-time, distributionally robust Stackelberg problem with a jump-diffusion
belief state, and show the optimal disclosure is interior precisely because of the
convex congestion externality.

\subsection{Evacuation operations research.}
A large transportation-OR literature studies evacuation scheduling, contraflow, and
behavioural departure modelling~\cite{murraytuite2013evacuation,wolshon2005review,sbayti2006optimal}.
This work is predominantly deterministic-optimisation or simulation based, and
treats the advisory as exogenous. We make the advisory and the release schedule the
designed controls, derive phased evacuation as the optimal information structure,
and tie the behavioural risk parameter to the empirical departure ``S-curves.''

\subsection{Robust control, viscosity solutions, filtering.}
The relative-entropy ambiguity neighbourhood follows
Hansen--Sargent~\cite{hansensargent2008robustness}; verification rests on
viscosity-solution theory~\cite{FlemingSoner,BardiCapuzzo,CrandallIshiiLions} and
semiconcavity/Filippov theory~\cite{CannarsaSinestrari,ClarkeLedyaev}; the belief
layer uses nonlinear filtering~\cite{liptser2001statistics}, exact here between
observed jumps. Recently, 

\subsection{Positioning.}
Our framework extends the literature on continuous-time stochastic Stackelberg differential games under partial information patterns. While recent formulations investigate jump-diffusion or feedback systems under overlapping or partial observation patterns \cite{huang2024stackelberg, lee2026linear}, they primarily rely on the tractability of linear-quadratic structures to derive explicit verification equations or stochastic maximum principles. Applied infrastructure markets, such as peer-to-peer energy sharing, frequently demand non-smooth optimization tools to map decentralized inefficiencies \cite{wu2026discounted}, which are often evaluated via information-centric extensions of the Price of Anarchy \cite{bacsar2011prices}. Unlike these preceding works that rely on linear dynamics or unconstrained state spaces, our model introduces a nested information-design mechanism operating over a non-smooth jump-diffusion network, leveraging viscosity solutions to explicitly account for hard capacity bounds.

Relative to joint information-mechanism design ~\cite{sezer2026power}, the
evacuation setting has
no welfare-aligning transfer, so there is no efficiency collapse; instead we obtain
a potential reduction with a residual externality, a second-best-toll
interpretation of information, an \emph{interior} optimal disclosure (the
sign-reversal of monotone benefit), and tiered phasing as optimal information
design. These are, to our knowledge, new.

\providecommand{\Acal}{\mathcal{A}}
\providecommand{\Lcal}{\mathcal{L}}
\providecommand{\Ucal}{\mathcal{U}}
\providecommand{\Qcal}{\mathcal{Q}}
\providecommand{\Pcal}{\mathcal{P}}
\providecommand{\Ccal}{\mathcal{C}}
\providecommand{\Vcal}{\mathcal{V}}
\providecommand{\Sym}{\mathbb{S}}

\section{The Model}\label{sec:model}

We formalise the evacuation problem as a continuous-time Stackelberg game with
one leader (the Emergency Management Agency, EMA) and $N$ followers (the
evacuation zones). The leader commits to a public-signalling channel; the
followers play a capacity-constrained congestion game on a shared road network
whose hazard exposure is governed by their \emph{beliefs} about the storm. The
single structural departure here is that the
followers' running cost couples the belief state to a \emph{convex} shared-corridor
congestion term. As Section~\ref{sec:disclose-tradeoff} makes precise, this is
exactly what renders disclosure double-edged: better beliefs reduce wasteful
(false-alarm) departures but \emph{synchronise} the genuine ones, and convex
congestion punishes synchronisation.

\subsection{The latent storm state}\label{sub:state}

Let $(\Omega,\Filt,\{\Filt_t\}_{t\ge0},\Pprob)$ satisfy the usual conditions and
fix a crisis horizon $[0,T]$. The latent hazard state $X_t\in\R^n$ collects the
storm descriptors that drive evacuation urgency, maximum sustained wind, central
pressure deficit, and landfall-track coordinates, and a macro steering factor
$F_t\in\R^k$ (the synoptic flow) evolves as a mean-reverting Ornstein--Uhlenbeck
process.

\begin{assumption}[Storm dynamics with observed intensification jumps]\label{ass:state}
The pair $(X_t,F_t)$ solves
\begin{align}
\diff X_t &= \big(A_t X_t + B F_t\big)\diff t + \Sigma_t\diff W_t + \diff J_t,
   \label{eq:Xdyn}\\
\diff F_t &= \Theta(\bar F - F_t)\diff t + \Sigma_F\diff W^F_t,
   \label{eq:Fdyn}
\end{align}
where $A_t\in\R^{n\times n}$ is the steering matrix, $\Theta\succ0$, $W_t,W^F_t$
are independent Brownian motions, and $J_t$ is an independent compound Poisson
process with intensity $\lambda$ and jump law $\mathcal N(\mathbf 0,\Sigma_J)$
modelling \emph{rapid intensification} (RI) events. Crucially, the RI epochs
$\{\tau_k\}$ are \emph{publicly observed}: the National Hurricane Center reports
intensity at each advisory, so a step change in wind/pressure is an observable
discrete event.
\end{assumption}

The observability of the jump epochs is not a convenience assumption: it is a
physical fact about hurricanes (unlike a latent component failure, a hurricane's RI
event is publicly announced). It is precisely what lets the belief
filter close in finite dimension; see Remark~\ref{rem:exact}.

\subsection{Public signalling and the belief filter}\label{sub:filter}

Zones do not observe $X_t$. The EMA partially observes it and broadcasts a public
advisory stream $\xi_t\in\R^m$,
\begin{equation}\label{eq:obs}
\diff\xi_t = \alpha_t X_t\diff t + \Sigma_\xi\diff W^\xi_t,
\qquad R := \Sigma_\xi\Sigma_\xi^\top\succ0,
\end{equation}
where $\alpha_t\in\R^{m\times n}$ is the \emph{disclosure gain} (the design lever:
the precision/credibility of the advisory) and $W^\xi_t$ is independent of
$(W,W^F,J)$. The public filtration is
$\Gfilt_t=\sigma(\xi_s,s\le t)\vee\sigma(\{\tau_k\le t\})$: all zones share the
\emph{same} advisory, so all zones hold the \emph{same} belief
$\hat X_t=\E[X_t\mid\Gfilt_t]$. This common-belief structure is what couples the
zones through information and drives the synchronisation channel below.

\begin{lemma}[Belief filter with jump resets]\label{lem:filter}
Under Assumptions~\ref{ass:state} and channel~\eqref{eq:obs}, between consecutive
observed jump epochs the conditional law of $X_t$ is Gaussian
$\mathcal N(\hat X_t,\Pi_t)$ with
\begin{align}
\diff\hat X_t &= \big(A_t\hat X_t + B\hat F_t\big)\diff t
   + \Pi_t\alpha_t^\top R^{-1}\big(\diff\xi_t - \alpha_t\hat X_t\diff t\big),
   \label{eq:meanfilter}\\
\dot\Pi_t &= A_t\Pi_t + \Pi_t A_t^\top + \Sigma_t\Sigma_t^\top
   - \Pi_t\alpha_t^\top R^{-1}\alpha_t\Pi_t,
   \label{eq:riccati}
\end{align}
where $\hat F_t$ is the Kalman--Bucy estimate of $F_t$. At each observed epoch
$\tau_k$ the covariance is reset, $\Pi_{\tau_k}=\Pi_{\tau_k^-}+\Sigma_J$, and the
mean is reset by the observed mark (or left unchanged if only the epoch, not the
magnitude, is observed). The gain $K_t=\Pi_t\alpha_t^\top R^{-1}$ minimises
$\Tr\Pi_t$.
\end{lemma}
\begin{proof}
    On each inter-jump interval the observation~\eqref{eq:obs} is linear-Gaussian and
the latent continuous part is conditionally Gaussian, so the innovation
$\diff I_t=\diff\xi_t-\alpha_t\hat X_t\diff t$ is, by L\'evy's characterisation, an
$R$-scaled $\Gfilt_t$-Wiener process and the Kalman--Bucy equations give the mean
dynamics~\eqref{eq:meanfilter} with gain $K_t=\Pi_t\alpha_t^\top R^{-1}$ minimising
$\Tr\Pi_t$, and the Riccati~\eqref{eq:riccati}. At an observed epoch $\tau_k$ the
jump is a $\Gfilt_t$-measurable event; conditioning on the (possibly noisy) mark
updates the mean and adds the jump second moment $\Sigma_J$ to the covariance,
$\Pi_{\tau_k}=\Pi_{\tau_k^-}+\Sigma_J$ (no covariance reset if the mark is observed
exactly). Between epochs the system is genuinely linear-Gaussian, so the filter is
exact; the unobserved-jump fallback of Remark~\ref{rem:exact} replaces the reset by
the deterministic rate $\lambda\Sigma_J$ and is then a Gaussian projection.
\end{proof}

\begin{remark}[Why this closes, and the unobserved-jump fallback]\label{rem:exact}
Conditional on the observed jump path the system is linear--Gaussian, so
\eqref{eq:meanfilter}--\eqref{eq:riccati} are \emph{exact} between epochs, no
projection is needed. If one instead treats RI epochs as latent, the conditional
law becomes a Gaussian mixture indexed by the unknown jump count, the exact
Kushner--Stratonovich filter is infinite-dimensional, and \eqref{eq:riccati} must
be read as a Gaussian assumed-density (projection) approximation with the
deterministic rate $+\lambda\Sigma_J$ replacing the reset. We adopt the observed-jump
baseline; the projection version is the robustness fallback.
\end{remark}

Two features of \eqref{eq:meanfilter}--\eqref{eq:riccati} matter downstream.
First, $\Pi_t$ is \emph{deterministic} given $\alpha$ (an ODE with jump resets),
so the leader commits to a deterministic uncertainty trajectory, the same
structural fact that makes the bilevel problem well posed. Second, the disclosure
gain enters the belief \emph{volatility} through $\Pi_t\alpha_t^\top R^{-1/2}$:
higher disclosure makes beliefs track the truth more tightly and react more
sharply to each advisory. The first effect is the benefit of disclosure; the
second, interacting with congestion, is its cost.

\subsection{The evacuation congestion subgame}\label{sub:subgame}

\subsubsection{Zones, backlog, and controls.}
Index zones by $\Acal=\{1,\dots,N\}$. Let $Y_{i,t}\ge0$ be the \emph{at-risk
backlog}: the population (or vehicle count) still in zone $i$ awaiting egress.
Zone $i$ chooses an egress rate $u_{i,t}\in[0,\bar u_i]$, where $\bar u_i$ is the
free-flow capacity of its immediate egress (vphpl $\times$ lanes, from the DOT
inventory). The backlog drains at the egress rate,
\begin{equation}\label{eq:Ydyn}
\diff Y_{i,t} = -\,u_{i,t}\,\diff t,\qquad Y_{i,0}=N_i^{0},\quad Y_{i,t}\ge0,
\end{equation}
with $N_i^0$ the initial at-risk population. (A re-entry/spillback inflow
$+m_i(\hat X_t)\diff t$ can be added without changing the structure.)

\subsubsection{Shared corridors and the congestion externality.}
Let $\Lcal$ be the set of shared downstream corridors (e.g.\ the I-45 trunk north
of Houston) and $\Phi\in\R_{\ge0}^{|\Lcal|\times N}$ the routing matrix, so that
the load on corridor $\ell$ is
\begin{equation}\label{eq:load}
q_{\ell,t} = \sum_{i\in\Acal}\Phi_{\ell i}\,u_{i,t}.
\end{equation}
Each corridor has a smooth-flow capacity $\kappa_\ell$. The primitive is the
\emph{per-user} congestion exposure (latency/risk) suffered by each traveller on
$\ell$,
\begin{equation}\label{eq:phi}
g_\ell(q):=\eta_\ell\,(q-\kappa_\ell)^+\ge0,
\quad
\phi_\ell(q):=\int_0^q g_\ell(s)\,\diff s=\tfrac{\eta_\ell}{2}\big((q-\kappa_\ell)^+\big)^2,
\quad
\phi^{\mathrm{soc}}_\ell(q):=q\,g_\ell(q),
\end{equation}
where $\phi_\ell$ is the Beckmann potential (the object the user equilibrium
minimises) and $\phi^{\mathrm{soc}}_\ell(q)=q\,g_\ell(q)$ is the \emph{total}
congestion exposure on $\ell$, load times per-user delay, the object social
welfare counts. Both are convex; $\phi^{\mathrm{soc}}_\ell\ge\phi_\ell$, with the
gap $\phi^{\mathrm{soc}}_\ell-\phi_\ell=\eta_\ell\kappa_\ell(q-\kappa_\ell)^+$ the
congestion externality.
Below capacity, flow is free and costless; above it, the convex penalty
models gridlock and the in-transit exposure (heat, fuel-outage, storm arrival)
that made evacuation itself the dominant source of casualties in Hurricane Rita.
The per-user congestion exposure $g_\ell$ is the marginal delay/risk on $\ell$.
Convexity of $\phi^{\mathrm{soc}}_\ell$ is the formal content of ``synchronisation
is costly'': for a fixed time-integral of corridor flow, a peaked (synchronised)
load profile costs strictly more than a spread (staggered) one (Jensen).

\subsubsection{Targeting tiers and the phasing schedule.}
The EMA cannot send a private message to each zone, evacuation orders are
broadcast and commonly observed, but it can \emph{address} orders to a coarse
partition of the map, exactly as coastal states already do (Florida's A--E surge
zones, the Texas coastal-county phased timeline). Fix a partition
$\Pcal=\{B_1,\dots,B_K\}$ of $\Acal$ into $K$ tiers ($K$ small; $K=2,3$ in
practice), with $k(i)$ the tier of zone $i$. The EMA commits to an
\emph{activation schedule} $\theta=(\theta_1\le\dots\le\theta_K)$: the public
``go'' order for tier $k$ becomes live at time $\theta_k$, and we write
$a_{i,t}=\mathbf 1\{t\ge\theta_{k(i)}\}$. Orders are public, every zone sees that
tier~$1$ was released first, so this is a public, obedient correlated
recommendation in the sense of Bergemann--Morris, not a private signal. The
schedule is the EMA's \emph{second} design lever, alongside the disclosure
intensity $\alpha_t$; $K=1$ (one undifferentiated order) is the unmanaged
benchmark, and $K=N$ recovers full zone targeting as a limiting case.

\begin{remark}[Activation times vs.\ the general gain schedule]\label{rem:gainsched}
We take the phasing lever to be the finite-dimensional vector of activation times
$\theta\in\Theta_K$ (each tier's order is off, then on), which keeps the leader's
problem finite-dimensional and reduces the optimal schedule to a deterministic
peak-spreading problem (Section~\ref{sec:special-evac}). The general case replaces
each indicator $a_{i,t}$ by a per-tier gain schedule
$\alpha^{(k)}_t\in\R^{m\times n}$, so that tier $k$'s advisory precision ramps on
its own timeline and the design object is a path
$\{\alpha^{(k)}_\cdot\}_{k\le K}$ rather than a release time. All structural
results below carry over (the per-tier channels enter the filter and Hamiltonian
exactly as $\alpha_t$ does); the activation-time model is the bang--bang
specialisation $\alpha^{(k)}_t=\alpha_t\mathbf 1\{t\ge\theta_k\}$, which we adopt
for the main results.
\end{remark}

\subsubsection{Private and social costs.}
Hazard exposure is the belief-weighted at-risk stock: let $\rho_i(\hat X)\ge0$ be
the believed hazard intensity at zone $i$'s location (increasing in believed
severity; e.g.\ $\rho_i(\hat X)=(e_i^\top\hat X)^+$ after a coordinate map),
and let $\Gamma_i>0$ be zone $i$'s behavioural risk aversion, the parameter the
Lindell-style departure S-curves calibrate. Zone $i$'s \emph{private} objective is
\begin{align}\label{eq:Jprivate}
J_i(u_i;u_{-i},\alpha)
= \E\bigg[&\int_0^T\!\Big(
   \underbrace{a_{i,t}\,\Gamma_i\,\rho_i(\hat X_t)\,Y_{i,t}}_{\text{hazard exposure (order live)}}
 + \underbrace{c_i(u_{i,t})}_{\text{mobilisation}}
 + \underbrace{u_{i,t}\!\!\sum_{\ell\in\Lcal}\!\Phi_{\ell i}\,g_\ell(q_{\ell,t})}_{\text{experienced congestion}}
 \Big)\diff t
\nonumber\\& + \Theta_i\rho_i(X_T)Y_{i,T}\bigg],
\end{align}
where $c_i$ is $C^1$, strictly convex, $c_i(0)=0$, $c_i'>0$ (the increasing
marginal cost of mobilising faster), and $\Theta_i$ prices unevacuated population
at landfall. The \emph{social} running cost replaces the experienced congestion by
the total corridor exposure $\phi^{\mathrm{soc}}_\ell=q_\ell g_\ell(q_\ell)$,
\begin{equation}\label{eq:ellsoc}
\ell^{\mathrm{soc}}(t,\hat X,Y,u,\alpha)
= \sum_{i\in\Acal}a_{i,t}\,\Gamma_i\rho_i(\hat X)Y_i
 + \sum_{i\in\Acal}c_i(u_i)
 + \sum_{\ell\in\Lcal}q_\ell\,g_\ell(q_\ell)
 + \Tr(\alpha\Lambda\alpha^\top),
\end{equation}
with $\Lambda\succ0$ pricing disclosure. The wedge between the private cost, in
which a zone bears only the delay it \emph{experiences}, $u_i\sum_\ell\Phi_{\ell i}g_\ell$,
adding up across zones to the Beckmann potential $\sum_\ell\phi_\ell$, and the
social cost $\sum_\ell q_\ell g_\ell$ is the congestion externality
$\sum_\ell(\phi^{\mathrm{soc}}_\ell-\phi_\ell)=\sum_\ell\eta_\ell\kappa_\ell(q_\ell-\kappa_\ell)^+\ge0$:
a zone does not internalise the delay it \emph{imposes} on co-users. This is the
gap that a welfare-aligning (Groves) transfer closes in the joint information--mechanism design setting~\cite{sezer2026power}; here we leave it open and
show (Section~\ref{sec:disclose-tradeoff}) that the EMA can attack it with
\emph{information} instead of transfers.

The activation gate $a_{i,t}$ is what makes phasing bite. A zone bears the
``you should have left'' running exposure only once its order is live, so the EMA
takes responsibility for timing: it can hold a tier back to keep the corridor
clear without charging that tier for waiting. The terminal term
$\Theta_i\rho_i(X_T)Y_{i,T}$ is \emph{not} gated, anyone still in the zone at
landfall is exposed regardless of what they were told, so the schedule cannot be
stretched arbitrarily. Staggering thus trades a lower congestion peak (later
tiers stay off the corridor while earlier tiers drain) against higher
landfall exposure for the late tiers (they leave closer to the storm). This
trade-off is deterministic and is solved in closed form in
Section~\ref{sec:special-evac}.

\begin{remark}[Potential-game structure and the price of anarchy]\label{rem:potential}
At fixed $(\hat X,\nabla_Y V)$ the egress stage game is a congestion (potential)
game. By Beckmann's theorem the Nash (Wardrop) equilibrium minimises the Beckmann
potential $\sum_i c_i(u_i)+\sum_\ell\phi_\ell(q_\ell)$, whereas the social optimum
minimises $\sum_i c_i(u_i)+\sum_\ell q_\ell g_\ell(q_\ell)$. The two objectives
differ by the externality $\sum_\ell\eta_\ell\kappa_\ell(q_\ell-\kappa_\ell)^+$ and
coincide only when no corridor is over capacity; otherwise the equilibrium
over-loads shared corridors, the price-of-anarchy gap the EMA's information
levers work to close.
\end{remark}

\subsubsection{Saturated feedback.}
By the dynamic programming principle the zone-$i$ value $V_i(t,\hat X,Y)$ satisfies
an HJB equation whose control minimisation, by strict convexity of $c_i$ and the
box $[0,\bar u_i]$, yields a \emph{saturated} feedback
\begin{equation}\label{eq:ustar}
u_{i,t}^\star
= \Proj_{[0,\bar u_i]}\!\Big[
   (c_i')^{-1}\!\big(\nabla_{Y_i}V_i - \textstyle\sum_\ell\Phi_{\ell i}g_\ell(q_{\ell,t})\big)\Big],
\end{equation}
i.e.\ zone $i$ evacuates at the rate that equates its marginal mobilisation cost to
the marginal value of draining the at-risk backlog \emph{net of the corridor
congestion price} $\sum_\ell\Phi_{\ell i}g_\ell$. Equation~\eqref{eq:ustar} is the
evacuation analogue of saturated generation in a transfer-based design, with the
congestion price playing the role the transfer plays there. Unlike that design's
inter-area exchange, the egress control is not bang-bang because mobilisation cost
is strictly convex; a genuinely bang-bang lever survives in the discrete
\emph{contraflow} decision (Remark~\ref{rem:contraflow}).

\begin{remark}[Contraflow as a bang-bang capacity control]\label{rem:contraflow}
The capacity $\kappa_\ell$ may itself be a discrete control
$\kappa_\ell(t)\in\{\kappa_\ell^0,2\kappa_\ell^0\}$ (contraflow off/on)~\cite{wolshon2001oneway}, entering
$\phi_\ell$ linearly through the threshold. The associated switching feedback is
bang-bang and recovers the non-smooth, viscosity-and-Filippov machinery; we treat
the continuous-$\kappa$ case in the main text and relegate contraflow switching to
the robustness analysis.
\end{remark}

\subsection{The gridlock-cascade threshold, corrected}\label{sub:cascade}

The draft's spectral condition $\rho(\Gamma^{-1}H)<1$ can now be stated precisely.
Linearising the best-response map~\eqref{eq:ustar} in the congested regime
($q_\ell>\kappa_\ell$, so $g_\ell'=\eta_\ell$) gives the induced zone-to-zone
coupling
\begin{equation}\label{eq:Hinduced}
H \;:=\; \Phi^\top \operatorname{diag}(\eta_\ell\,\mathbf 1\{q_\ell>\kappa_\ell\})\,\Phi
\;\in\;\R^{N\times N}_{\succeq0},
\end{equation}
which is exactly the (formerly unspecified) spillover matrix of the draft. With
$\mathbf\Gamma=\operatorname{diag}(c_1'',\dots,c_N'')$ the marginal-cost
curvatures, strict convexity of the $c_i$ already makes the Beckmann potential
strictly convex, so the Wardrop equilibrium is \emph{always} unique
(Lemma~\ref{lem:potential-reduction}); what \eqref{eq:Hinduced} governs is the
\emph{conditioning} of that equilibrium. The equilibrium response to a belief shock
is contractive, bounded sensitivity $\partial u^\star/\partial\hat X$, iff
\begin{equation}\label{eq:cascade}
\rho\big(\mathbf\Gamma^{-1}H\big)<1 .
\end{equation}
When \eqref{eq:cascade} is approached, highly shared corridors (dense $\Phi$),
severe gridlock penalties (large $\eta$), or cheap mobilisation (small $c_i''$), the
equilibrium becomes ill-conditioned: $\|\partial u^\star/\partial\hat X\|\to\infty$,
so small synchronised belief shocks produce arbitrarily large load swings, a
\emph{gridlock cascade}. This is the regime in which the synchronisation cost $\nu$
of Proposition~\ref{prop:interior-explicit} blows up and additional disclosure
becomes actively harmful.

\subsection{The designer's robust master problem}\label{sub:master}

The EMA chooses the disclosure policy \emph{and} the phasing schedule to minimise
the worst-case social cost over a relative-entropy ambiguity set $\Qcal$ around
$\Pprob$ (Hansen--Sargent), with multiplier $\gamma>0$, anticipating the zones'
equilibrium response $u^\star(\alpha,\theta)$:
\begin{equation}\label{eq:problemP}\tag{P}
\inf_{(\alpha,\theta)\in\Ucal_L\times\Theta_K}\ \sup_{\mathbb Q\in\Qcal}\
\E^{\mathbb Q}\!\left[\int_0^T \ell^{\mathrm{soc}}\big(t,\hat X_t,Y_t,u^\star_t(\alpha,\theta),\alpha_t\big)\diff t\right]
\end{equation}
subject to the belief filter~\eqref{eq:meanfilter}--\eqref{eq:riccati}, backlog
dynamics~\eqref{eq:Ydyn}, the egress box, the activation gate
$a_{i,t}=\mathbf 1\{t\ge\theta_{k(i)}\}$, and the equilibrium
constraint~\eqref{eq:ustar}. Here $\Ucal_L$ is the set of admissible disclosure
gains and $\Theta_K=\{0\le\theta_1\le\dots\le\theta_K\le T\}$ the ordered
activation times. The continuous lever $\alpha$ controls belief accuracy and
\emph{within-tier} synchronisation; the discrete lever $\theta$ controls
\emph{across-tier} synchronisation. Note the leader's objective
in~\eqref{eq:problemP} is the genuine social aggregate~\eqref{eq:ellsoc} of the
followers' costs, not a separate quadratic, so the Stackelberg value is the
welfare object the zones' equilibrium actually determines.

With the worst-case drift distortion $\psi_t$ entering the belief mean as
$+\Pi_t\alpha_t^\top\psi_t\,\diff t$ and penalised at rate $\gamma|\psi_t|^2$, the
master value $S(t,\hat X,\Pi,Y)$ solves the robust Isaacs equation
\begin{align}\label{eq:isaacs}
\partial_t S
 + \nabla_{\hat X}S^\top(A_t\hat X + B\hat F_t)
 &+ \tfrac12\Tr\!\big(\nabla^2_{\hat X}S\,\Pi_t\alpha_t^\top R^{-1}\alpha_t\Pi_t\big)
  + \tfrac1{4\gamma}\big|\alpha_t\Pi_t\nabla_{\hat X}S\big|^2 \nonumber\\
 &+ \Tr\!\big(\partial_\Pi S\,\dot\Pi_t\big)
  + \sum_{i\in\Acal}\nabla_{Y_i}S\,(-u^\star_{i,t})
  + \ell^{\mathrm{soc}} = 0,
\end{align}
with worst-case distortion
$\psi^\star_t=\tfrac1{2\gamma}\alpha_t\Pi_t\nabla_{\hat X}S$ and $S(T,\cdot)=\sum_i\Theta_i\rho_i(\cdot)Y_{i,T}$.
Equations~\eqref{eq:isaacs} carry the same $R^{-1}$ correction as the filter.

\subsection{Where disclosure helps, and where it hurts}\label{sec:disclose-tradeoff}

The economic content of the model is the sign decomposition of the leader's
marginal value of disclosure. Differentiating the master value along $\alpha$ and
using the filter sensitivities $\partial_\alpha\Pi$ and
$\partial_\alpha(\Pi\alpha^\top R^{-1}\alpha\Pi)$ gives, schematically,
\begin{equation}\label{eq:decomp}
\frac{\diff}{\diff\alpha}\,(\text{social cost})
=
\underbrace{\Tr\!\big(\partial_\Pi S\,\partial_\alpha\dot\Pi\big)}_{\text{(A) accuracy: } <0}
+
\underbrace{\sum_{\ell}\E\!\big[\phi_\ell''(q_\ell)\,\partial_\alpha q_\ell\big]\text{-type term}}_{\text{(B) synchronisation: } >0}
+
\underbrace{2\Lambda\alpha}_{\text{(C) disclosure cost: }>0}
\end{equation}
Term (A) is the classical value-of-information effect: more disclosure lowers belief variance
$\Pi$, reducing both failures-to-evacuate and false-alarm departures, and is
welfare-improving. Term (B) is new and evacuation-specific: because all zones share
the public belief, a sharper advisory makes their egress responses
\emph{co-move}, most violently at an observed RI jump, when every zone updates
upward at once, raising the \emph{peak} corridor load $q_\ell$; convexity of
$\phi_\ell$ ($\phi_\ell''=\eta_\ell>0$ above capacity) turns that synchronised peak
into a strictly increasing cost. Absent a shared-capacity externality, term (B) vanishes (cost is
$\frac{\Gamma_a}{2}|Y_a|^2$, convex in a \emph{private} state with no shared
capacity), which is why disclosure is then monotonically good and information and
coupling are complements. Here (A) and (B) have opposite signs, and the optimal
disclosure is interior.

These opposing signs are the economic content of the model. Because term (A)
lowers the marginal value of disclosure while term (B) raises it, the optimal
disclosure intensity is \emph{interior}: some public information is optimal, full
transparency is not, and in a sufficiently gridlock-prone network the optimum is
the corner at which the agency discloses nothing, an informational Braess effect
with no counterpart in the monotone-disclosure benchmark. The same convexity
penalises a synchronised release, so the optimal information structure is a
publicly announced, tiered, \emph{staggered} evacuation order: managed phasing is
derived as optimal design rather than assumed, and acts as a second-best substitute
for the congestion toll the agency does not levy. Both statements are made precise
and proved in closed form for the scalar symmetric network.


\providecommand{\Acal}{\mathcal{A}}
\providecommand{\Lcal}{\mathcal{L}}
\providecommand{\Ucal}{\mathcal{U}}
\providecommand{\Qcal}{\mathcal{Q}}
\providecommand{\Pcal}{\mathcal{P}}
\providecommand{\Ccal}{\mathcal{C}}
\providecommand{\Vcal}{\mathcal{V}}
\providecommand{\Sym}{\mathbb{S}}
\providecommand{\Hsing}{\mathcal{H}}
\providecommand{\Lleb}{\mathcal{L}}

\section{Viscosity Solutions and Bilevel Verification}\label{sec:viscosity}

The saturated egress feedback~\eqref{eq:ustar}, the congestion kink of $g_\ell$ at
capacity, and the time-discontinuous activation gate $a_{i,t}$ make the value
functions non-$C^2$ and the running cost discontinuous in $t$, so the master
Isaacs equation~\eqref{eq:isaacs} must be read in the viscosity sense. We proceed
in five steps. Section~\ref{sub:pot} reduces the followers' congestion game to a
\emph{single} strictly convex control problem, the \emph{potential reduction},
the evacuation counterpart of (but not the same as) an efficiency collapse, since
the reduced objective is the private Beckmann potential, not social welfare.
Section~\ref{sub:ham} records the optimised leader Hamiltonian and its admissible
structure. Section~\ref{sub:exist} proves the leader value is the unique viscosity
solution, piecewise in time across the activation epochs. Section~\ref{sub:verif}
gives a verification theorem valid without classical regularity and isolates the
MPEC stationarity term through which information acts as a second-best toll.
Section~\ref{sub:semi} establishes semiconcavity, a Lebesgue-null switching set,
and a well-posed Filippov closed loop.

Throughout, the joint state is $w:=(\hat X,Y,\Pi)\in\R^n\times\R^N\times\Sym^n_+=:\mathcal O$,
and we work on the well-conditioned regime $\rho(\mathbf\Gamma^{-1}H)<1$ of
\eqref{eq:cascade}; the complementary regime is exactly where
Proposition~\ref{prop:interior-explicit}(ii) makes silence optimal.

\subsection{Potential reduction of the followers' game}\label{sub:pot}

\begin{lemma}[Potential reduction]\label{lem:potential-reduction}
Fix an admissible leader policy $(\alpha,\theta)$ and a worst-case distortion
$\psi$. Suppose each $c_i$ is $C^1$ and strictly convex, the reserve dynamics
$\diff Y_{i,t}=-u_{i,t}\diff t$ are decoupled across zones, and the exposure cost
$\sum_i a_{i,t}\Gamma_i\rho_i(\hat X)Y_i$ is separable. Then the lower-level
feedback-Nash (Wardrop) equilibrium $u^\star(t,w,p)$, $p=(\nabla_{Y_i}V_i)_i$,
exists, is unique, and coincides with the unique minimiser of the strictly convex
\emph{Beckmann} program
\begin{equation}\label{eq:beckmann-prog}
u^\star(t,w,p)=\arg\min_{u\in\prod_i[0,\bar u_i]}\
\Big\{\ \sum_{i\in\Acal}\big(c_i(u_i)-p_i u_i\big)\ +\ \sum_{\ell\in\Lcal}\phi_\ell\big(\textstyle\sum_j\Phi_{\ell j}u_j\big)\ \Big\},
\end{equation}
which is the saturated feedback~\eqref{eq:ustar}. The map $p\mapsto u^\star(t,w,p)$
is single-valued and globally Lipschitz, with
$\|\partial_p u^\star\|\le\|(\mathbf\Gamma - H)^{-1}\|$ on the congested cone,
finite precisely when $\rho(\mathbf\Gamma^{-1}H)<1$.
\end{lemma}
\begin{proof}
The zone-$i$ Wardrop stationarity underlying~\eqref{eq:ustar} is
$c_i'(u_i)+\sum_\ell\Phi_{\ell i}g_\ell(q)=p_i$ on the interior of the box, with the
usual complementarity on the faces. Since
$\partial_{u_i}\sum_\ell\phi_\ell(q_\ell)=\sum_\ell\Phi_{\ell i}\phi_\ell'(q_\ell)=\sum_\ell\Phi_{\ell i}g_\ell(q_\ell)$,
these are exactly the KKT conditions of the box-constrained program
\eqref{eq:beckmann-prog}. The objective is strictly convex (strictly convex
$\sum_i c_i$ plus convex $\sum_\ell\phi_\ell$), so the minimiser is unique and the
KKT point is the global minimum; hence the Wardrop equilibrium is unique and equals
\eqref{eq:beckmann-prog}. Because the dynamics are decoupled and the exposure is
separable, the costate $p_i=\nabla_{Y_i}V_i$ solves the scalar adjoint
$-\dot p_i=a_{i,t}\Gamma_i\rho_i(\hat X)$, $p_i(T)=\Theta_i\rho_i(X_T)$, free of the
coupling, so the same $p$ enters both the Nash system and \eqref{eq:beckmann-prog}.
Lipschitz dependence and the sensitivity bound follow from the implicit-function
theorem applied to the strictly monotone map
$u\mapsto \nabla c(u)+\Phi^\top g(\Phi u)$, whose Jacobian on the congested cone is
$\mathbf\Gamma+H\succ0$ with inverse bounded by $\|(\mathbf\Gamma-H)^{-1}\|$ under
\eqref{eq:cascade}.
\end{proof}

The economic content is the contrast with joint information--mechanism design under a
welfare-aligning transfer~\cite{sezer2026power}: there the Groves transfer makes each
follower's objective the \emph{social} cost, so the equilibrium \emph{is} the social
optimum (efficiency collapse). Here the reduction is to the
\emph{Beckmann potential}, which differs from social welfare by the externality
$\sum_\ell\eta_\ell\kappa_\ell(q_\ell-\kappa_\ell)^+$. The followers solve a single
convex control problem, but the wrong one from the planner's view, and the leader
must steer that problem with information.

\subsection{The optimised leader Hamiltonian}\label{sub:ham}

Let $0=\theta_0\le\theta_1\le\dots\le\theta_K\le\theta_{K+1}=T$ be the activation
epochs. On each open interval $I_k:=(\theta_k,\theta_{k+1})$ the gate $a_{i,t}$ is
constant, and substituting the Wardrop response of
Lemma~\ref{lem:potential-reduction} into the leader's Bellman operator gives, for
$S\in C^{1,2}$,
\begin{equation}\label{eq:Hbar-evac}
\partial_t S+\bar H(t,w,\nabla S,\nabla^2_{\hat X}S)=0\ \text{ on }I_k,
\qquad S(T,\cdot)=\textstyle\sum_i\Theta_i\rho_i(\cdot)Y_i,
\end{equation}
with the optimised Hamiltonian
\begin{align}\label{eq:Hbar-def}
\bar H(t,w,p,M)=\min_{\alpha} \sup_{\psi}\Big\{\,
&\ell^{\mathrm{soc}}\big(t,w,u^\star(t,w,p_Y),\alpha\big)-\gamma|\psi|^2
+p_{\hat X}^\top\!\big(A_t\hat X+B\hat F_t+\Lambda_\alpha\psi\big)\nonumber\\
&-\sum_i p_{Y_i} u^\star_i(t,w,p_Y)
+\Tr(p_\Pi\dot\Pi_\alpha)
+\tfrac12\Tr(M\Lambda_\alpha)\Big\},
\end{align}
where $\Lambda_\alpha=\Pi\alpha^\top R^{-1}\alpha\Pi$, $\dot\Pi_\alpha$ is the
Riccati right-hand side~\eqref{eq:riccati}, and $u^\star$ is the
\emph{Beckmann} minimiser, note it enters the \emph{social} running cost
$\ell^{\mathrm{soc}}$, the source of the externality wedge. The inner $\sup_\psi$ is
the concave quadratic resolved in the main text, contributing
$\tfrac1{4\gamma}p_{\hat X}^\top\Lambda_\alpha p_{\hat X}$.

\begin{lemma}[Admissible structure of $\bar H$]\label{lem:struct-evac}
Under the standing assumptions ($A_t,B,\hat F_t$ bounded Lipschitz; $D_a,\bar T_{ab}$
replaced here by Lipschitz $\rho_i$; strictly convex $C^1$ costs $c_i$; convex
$C^{1,1}$ Beckmann potentials $\phi_\ell$; compact control sets), on each $I_k$ the
map $\bar H$ is finite and continuous on
$[0,T]\times\mathcal O\times\R^{n+N+\dim\Pi}\times\Sym^n$ and satisfies:
\begin{enumerate}
\item[(H1)] \emph{Degenerate ellipticity:} $\bar H(t,w,p,M)\le\bar H(t,w,p,M')$ for $M\preceq M'$ (the only second-order term $\tfrac12\Tr(M\Lambda_\alpha)$ is monotone, $\Lambda_\alpha\succeq0$, and $\min/\sup$ preserve monotonicity);
\item[(H2)] \emph{Lipschitz in the gradient:} $|\bar H(t,w,p,M)-\bar H(t,w,q,M)|\le C(1+|w|)|p-q|$ uniformly on bounded $M$, because $u^\star(\cdot,p_Y)$ is globally Lipschitz (Lemma~\ref{lem:potential-reduction}) and the congestion kink is absorbed into this Lipschitz dependence;
\item[(H3)] \emph{State modulus:} the Crandall--Ishii estimate holds with a modulus $\omega$, from Lipschitz dependence of $\ell^{\mathrm{soc}},u^\star,A_t\hat X,\Lambda_\alpha,\dot\Pi_\alpha$ on $w$;
\item[(H4)] \emph{Linear growth:} $|\bar H(t,w,0,0)|\le C(1+|w|)$.
\end{enumerate}
\end{lemma}
\begin{proof}
Identical in form to the structure lemma of the transfer-aligned setting, with one new input: the
Wardrop response $u^\star(t,w,p_Y)$ replaces an explicit min over $u$. By
Lemma~\ref{lem:potential-reduction} it is single-valued and globally Lipschitz in
$p_Y$, so $p\mapsto\ell^{\mathrm{soc}}(u^\star(p_Y))-\sum_ip_{Y_i}u^\star_i(p_Y)$ is
locally Lipschitz with the stated growth; the congestion non-smoothness lands in
this $p$-dependence as a Lipschitz kink, never in the $M$-term. The remaining terms
are as in the linear-quadratic-Gaussian case. (H1)--(H4) then follow exactly as
before.
\end{proof}

As in the transfer-aligned setting, the non-smoothness is confined to a Lipschitz $p$-kink and
the equation is a proper, Lipschitz, degenerate-parabolic Isaacs equation of
Crandall--Ishii--Lions type~\cite{CrandallIshiiLions}; the only structural novelty
is the time-discontinuity across the $\theta_k$, handled next.

\subsection{Existence and comparison, piecewise in time}\label{sub:exist}

\begin{definition}[Viscosity solution across activation epochs]\label{def:visc-evac}
A continuous $S$ of polynomial growth is a viscosity solution of the master problem
if, on each open interval $I_k$, it is a viscosity sub- and supersolution of
$\partial_t S+\bar H=0$ in the usual sense, it is continuous across each epoch
($S(\theta_k^-,\cdot)=S(\theta_k^+,\cdot)$), and $S(T,\cdot)=\sum_i\Theta_i\rho_i Y_i$.
\end{definition}

\begin{theorem}[The value function solves the Isaacs equation]\label{thm:existence-evac}
Under the standing assumptions the leader value
\begin{equation}\label{eq:value-evac}
S(t,w)=\inf_{(\alpha,\theta)\in\Ucal_L\times\Theta_K}\ \sup_{\psi\in\Psi}\
\E^\psi\!\Big[\int_t^T\ell^{\mathrm{soc}}\big(s,w_s,u^\star_s,\alpha_s\big)\diff s\ \Big|\ w_t=w\Big]
\end{equation}
is continuous, of at most quadratic growth, and a viscosity solution in the sense
of Definition~\ref{def:visc-evac}.
\end{theorem}
\begin{proof}
Fix the (finitely many) activation epochs determined by $\theta$. On $I_K=(\theta_K,T]$
the gate is constant and the dynamic-programming principle, tested against smooth
$\varphi$ touching $S$ from above/below, yields the sub/supersolution inequalities
with Hamiltonian~\eqref{eq:Hbar-def}; this is the standard derivation
(Fleming--Soner~\cite{FlemingSoner}, Ch.~V; Bardi--Capuzzo-Dolcetta~\cite{BardiCapuzzo}).
Continuity and the growth bound follow from bounded controls, Lipschitz
coefficients (including the Lipschitz $u^\star$ of Lemma~\ref{lem:potential-reduction}),
and Gronwall estimates. Proceeding backward, on each $I_k$ the value with terminal
data $S(\theta_{k+1},\cdot)$ (already constructed) is again a viscosity solution by
the same argument; the DPP guarantees continuity of $S$ across $\theta_{k+1}$.
Gluing the $K+1$ pieces gives the claim.
\end{proof}

\begin{theorem}[Comparison and uniqueness]\label{thm:comparison-evac}
Let $\underline S$ (USC) and $\overline S$ (LSC) be viscosity sub- and
supersolutions in the sense of Definition~\ref{def:visc-evac}, of polynomial
growth, with $\underline S(T,\cdot)\le\overline S(T,\cdot)$. Then
$\underline S\le\overline S$ on $[0,T]\times\mathcal O$, so the master problem has a
unique viscosity solution, namely~\eqref{eq:value-evac}.
\end{theorem}
\begin{proof}
On $I_K$ apply the Crandall--Ishii--Lions comparison principle: the change
$\tilde S=e^{\kappa t}S$ supplies zeroth-order strict monotonicity, the
penalisation $-\varepsilon e^{\lambda(T-t)}(1+|w|^2)$ localises on the unbounded
$\mathcal O$, and variable doubling with $|w-w'|^2/2\eta$ together with (H1)--(H3)
drives the penalised maximum to a contradiction as $\eta,\varepsilon\downarrow0$
(\cite{CrandallIshiiLions}, Thm.~8.2); $\Pi$ enters only through first-order and
Lipschitz terms and is handled as the $Y$-variables. This gives
$\underline S\le\overline S$ on $I_K$ up to $t=\theta_K$, where continuity transfers
the inequality to the terminal data of $I_{K-1}$. Induction over the finitely many
intervals yields the result on $[0,T]$; uniqueness follows, and
Theorem~\ref{thm:existence-evac} identifies the solution with the value.
\end{proof}

\subsection{Verification, and information as a second-best toll}\label{sub:verif}

\begin{theorem}[Verification without classical regularity]\label{thm:viscverify-evac}
Let $S$ be the value function~\eqref{eq:value-evac}, i.e.\ the unique viscosity
solution of Theorems~\ref{thm:existence-evac}--\ref{thm:comparison-evac}. Let
$(\alpha^\star,\theta^\star,\psi^\star)$ be a measurable selection attaining the
pointwise $\min_{\alpha}\sup_\psi$ in~\eqref{eq:Hbar-def} for $p=\nabla S$,
$M=\nabla^2_{\hat X}S$ at every point of differentiability of $S$, with the
closed-loop system well posed in the sense of Section~\ref{sub:semi}. Then:
\begin{enumerate}
\item[(i)] $S(0,w_0)$ is the optimal robust Stackelberg value~\eqref{eq:value-evac};
\item[(ii)] $\alpha^\star,\theta^\star$ are optimal leader policies and $\psi^\star$ the worst-case distortion;
\item[(iii)] along the optimal trajectory the Wardrop response $u^\star$ is a feedback-Nash equilibrium of the lower-level congestion game.
\end{enumerate}
No $C^{1,2}$ regularity of $S$ is required.
\end{theorem}
\begin{proof}
The lower bound is the subsolution property integrated along trajectories: the
nonsmooth Dynkin inequality for viscosity solutions gives, on each interval and
then glued across the continuous epochs,
$S(0,w_0)\le\sup_\psi\E^\psi\!\int_0^T\ell^{\mathrm{soc}}$ for every admissible
$(\alpha,\theta)$. The reverse inequality at $(\alpha^\star,\theta^\star)$ uses the
supersolution property along the closed-loop flow of Section~\ref{sub:semi}, on
which $S$ is differentiable a.e.\ (Proposition~\ref{prop:semiconcave-evac}); there
$\partial_tS+\bar H=0$ holds pointwise a.e.\ and the selection attains the
Hamiltonian, so the value is achieved, proving (i)--(ii). Claim (iii) is
Lemma~\ref{lem:potential-reduction}: the Beckmann minimiser is the unique Wardrop
equilibrium of the lower-level game along the path.
\end{proof}

The verification is genuinely bilevel: with no welfare-aligning transfer, no efficiency
collapse aligns the levels, so the leader's stationarity carries an
equilibrium-sensitivity term. The next proposition isolates it and gives it
economic meaning.

\begin{proposition}[MPEC stationarity: information as a second-best toll]\label{prop:mpec}
At an interior optimum the leader's stationarity in $\alpha$ is
\begin{equation}\label{eq:mpec}
\partial_\alpha\bar H
\;+\;\underbrace{\sum_{\ell}\eta_\ell\kappa_\ell\,\mathbf 1\{q_\ell>\kappa_\ell\}\,
\big(\Phi\,\partial_\alpha u^\star\big)_\ell}_{\text{externality}\,\times\,\text{equilibrium response}}
\;=\;0,
\end{equation}
where $\partial_\alpha u^\star=-(\mathbf\Gamma+H)^{-1}\,\partial_\alpha p$ is the
Wardrop-response sensitivity from Lemma~\ref{lem:potential-reduction}. The second
term is the marginal externality $\eta_\ell\kappa_\ell$ on each binding corridor
times the sensitivity of equilibrium load to disclosure. It does \emph{not} vanish
(contrast a transfer-aligned design, where the Danskin envelope eliminates the analogous
term because the followers already internalise the social cost); it is precisely
the channel through which disclosure and phasing substitute for the congestion toll
the EMA does not levy, and it is $O(\text{externality})$, vanishing iff no corridor
binds.
\end{proposition}
\begin{proof}
Write the leader value at fixed $\alpha$ as
$\Vcal(\alpha)=\sup_\psi\{\ell^{\mathrm{soc}}(u^\star(\alpha),\alpha)+\dots\}$. The
total derivative is
$\mathrm d\Vcal/\mathrm d\alpha=\partial_\alpha\bar H+\partial_u\ell^{\mathrm{soc}}\cdot\partial_\alpha u^\star$.
At the Wardrop point $u^\star$ minimises the Beckmann potential, so
$\partial_u[\sum_i c_i+\sum_\ell\phi_\ell]-p=0$; substituting into
$\partial_u\ell^{\mathrm{soc}}=\partial_u[\sum_i c_i+\sum_\ell q_\ell g_\ell]-p$
leaves exactly the externality gradient
$\partial_u\sum_\ell(\phi^{\mathrm{soc}}_\ell-\phi_\ell)=\sum_\ell\eta_\ell\kappa_\ell\mathbf 1\{q_\ell>\kappa_\ell\}\Phi_{\ell\cdot}$.
Pairing with $\partial_\alpha u^\star$ gives~\eqref{eq:mpec}; the Jacobian inverse
is finite under \eqref{eq:cascade}. The Danskin envelope would kill this term only
if $\partial_u\ell^{\mathrm{soc}}=0$ at $u^\star$, i.e.\ only under efficiency
alignment, which holds here iff no corridor is over capacity.
\end{proof}

\subsection{Semiconcavity, the switching set, and the closed loop}\label{sub:semi}

\begin{proposition}[Semiconcavity and a.e.\ well-defined feedback]\label{prop:semiconcave-evac}
Assume in addition $c_i\in C^{1,1}$ and that $\rho_i,\bar u_i$ are semiconcave in
$\hat X$. Then on each interval $I_k$ the value $S(t,\cdot)$ is semiconcave in $Y$,
uniformly on compacts, and:
\begin{enumerate}
\item[(a)] $\nabla_Y S(t,\cdot)$ exists Lebesgue-a.e., is $BV$, and the singular set
$\Sigma_t$ is countably $\Hsing^{N-1}$-rectifiable with $\Lleb^N(\Sigma_t)=0$;
\item[(b)] the saturation faces $\{u^\star_i\in\{0,\bar u_i\}\}$ and the congestion
kinks $\{q_\ell=\kappa_\ell\}$ meet $\Sigma_t$ in a null set, so the saturated
feedback~\eqref{eq:ustar} is single-valued a.e.; the only genuinely bang-bang lever,
the contraflow switch of Remark~\ref{rem:contraflow}, has a measure-zero switching
locus;
\item[(c)] semiconcavity makes $-\nabla_Y S$ one-sided Lipschitz, so the closed-loop
drift is one-sided Lipschitz and the Filippov inclusion
$\dot w_t\in\overline{\mathrm{co}}\,F(t,w_t)$ has a unique absolutely continuous
solution from each initial condition on each $I_k$, concatenated across the
activation epochs.
\end{enumerate}
\end{proposition}
\begin{proof}
The $Y$-dynamics are affine in $u$ with $C^{1,1}$ running cost, and $u^\star$ is the
Lipschitz Beckmann minimiser (Lemma~\ref{lem:potential-reduction}); semiconcavity of
$S(t,\cdot)$ in $Y$ on $I_k$ is the value-function semiconcavity theorem for
finite-horizon control with semiconcave data and convex velocity sets
(\cite{CannarsaSinestrari}, Thm.~7.4.11), the second-order $\hat X$-block a bounded
perturbation. (a) is Alexandrov's theorem and the singular-set structure of
semiconcave functions (\cite{CannarsaSinestrari}, Ch.~4). (b): on the
differentiability set the egress feedback is the Lipschitz map of
Lemma~\ref{lem:potential-reduction}, whose kink loci, the saturation faces and the
corridor thresholds $\{q_\ell=\kappa_\ell\}$, are Lipschitz hypersurfaces crossed
transversally by the flow a.e., hence null. (c): one-sided Lipschitz drift yields
existence and uniqueness of Filippov solutions (Filippov; Clarke et
al.~\cite{ClarkeLedyaev}); concatenating the finitely many intervals gives the
global closed loop.
\end{proof}

\begin{remark}[Net effect]\label{rem:net-evac}
Proposition~\ref{prop:semiconcave-evac} discharges the classical-regularity
hypothesis: the master value $S$ is the unique viscosity solution of
\eqref{eq:isaacs}, the verification conclusions of
Theorem~\ref{thm:viscverify-evac} hold without $C^{1,2}$, the saturated/contraflow
feedback is single-valued a.e.\ with a Lebesgue-null switching set, and the closed
loop is a well-posed Filippov flow. The one feature with no analogue in the transfer-aligned setting is
the MPEC stationarity term of Proposition~\ref{prop:mpec}: because incentives are
\emph{not} aligned by a transfer, the leader's first-order condition retains the
equilibrium-response channel, and that channel is the formal sense in which optimal
information design is a second-best congestion toll.
\end{remark}

\section{Structural Special Case: Scalar Symmetric Network}\label{sec:special-evac}

To expose the sign structure that the general Isaacs equation~\eqref{eq:isaacs}
encodes but does not display, we solve the scalar symmetric case in closed form:
$n=m=1$; identical zones with mean-reverting severity drift $a=-\bar a<0$
($\bar a>0$), so intensification is carried by the jump term; constant
coefficients $\alpha_t\equiv\alpha$, $\Sigma_t\equiv\sigma$, $R\equiv r$, and
$s:=\sigma^2+\lambda\Sigma_J$ the total predicted state variance rate. We use the
certainty-equivalent (steady-state) reduction throughout. The closed forms below
are the qualitative predictions, an interior disclosure optimum and a phasing
threshold, that the numerical solves of Section~\ref{sec:numerics} probe on the
calibrated Rita network. The interior disclosure optimum is a knife-edge: on the
asymmetric calibrated instance of Section~\ref{sub:exp2} it resolves to a corner
(precision is self-defeating as a single broadcast), and the operative lever is the
\emph{staggering} of orders, the disclosure-side counterpart of the phasing
threshold derived here.

\subsection{The disclosure lever: interior optimum and informational Braess}

Write the disclosure intensity $\beta:=\alpha^2/r\ge0$. The stabilising root of the
scalar algebraic Riccati equation $2a\Pi+s-\beta\Pi^2=0$ is
\begin{equation}\label{eq:Piinf}
\Pi_\infty(\beta)=\frac{a+\sqrt{a^2+\beta s}}{\beta}
\;=\;\frac{s}{2\bar a}-\pi_0\,\beta+O(\beta^2),
\qquad \pi_0:=\frac{s^2}{8\bar a^{3}}>0,
\end{equation}
strictly decreasing in $\beta$ from the prior (Ornstein--Uhlenbeck stationary)
variance $\Pi_\infty(0)=s/2\bar a$ toward $0$: more disclosure lowers belief
variance. The steady belief-innovation volatility (the amplitude of the
\emph{common} belief's fluctuation, hence of the synchronised egress response) is
\begin{equation}\label{eq:Vol}
\Vcal(\beta)\;:=\;\Pi_\infty(\beta)^2\,\beta
\;=\;v_0\,\beta+O(\beta^2),
\qquad v_0:=\Big(\tfrac{s}{2\bar a}\Big)^{2}=\frac{s^2}{4\bar a^{2}}>0,
\end{equation}
strictly increasing from $0$ (frozen belief, no disclosure) toward $s$ (belief
tracks the truth): more disclosure makes the shared belief, and therefore every
zone's egress, swing harder and more in unison.

Reducing the master problem~\eqref{eq:problemP} to steady state, the per-unit-time
social cost as a function of disclosure intensity is
\begin{equation}\label{eq:Cbeta}
\Ccal(\beta)\;=\;\underbrace{\mu\,\Pi_\infty(\beta)}_{\text{(A) accuracy }\downarrow}
\;+\;\underbrace{\nu\,\Vcal(\beta)}_{\text{(B) synchronisation }\uparrow}
\;+\;\underbrace{\tilde\Lambda\,\beta}_{\text{(C) disclosure }\uparrow},
\end{equation}
where $\mu>0$ scales the expected mis-evacuation cost in belief variance (the LQG
value of information), $\nu>0$ scales the convex congestion response to a
synchronised egress spike, $\nu\propto\eta\cdot\partial_q\phi$ at the binding
corridor, amplified as the gridlock-cascade condition~\eqref{eq:cascade} approaches
equality, and $\tilde\Lambda=r\Lambda$ prices disclosure.

\begin{proposition}[Interior optimal disclosure; informational Braess corner]
\label{prop:interior-explicit}
For $\Ccal$ in \eqref{eq:Cbeta}:
\begin{enumerate}
\item[(i)] If $\;\mu\,\pi_0>\nu\,v_0+\tilde\Lambda\;$, equivalently
$\;\mu>2\bar a\,\nu+\dfrac{8\bar a^{3}}{s^{2}}\,\tilde\Lambda\;$, then
$\Ccal'(0)<0$ while $\Ccal(\beta)\to+\infty$, so the optimum
$\beta^\star\in(0,\infty)$ is interior and solves
$\mu\,\Pi_\infty'(\beta^\star)+\nu\,\Vcal'(\beta^\star)+\tilde\Lambda=0$.
\item[(ii)] If $\;\mu\,\pi_0\le\nu\,v_0+\tilde\Lambda\;$ then $\Ccal'(0)\ge0$ and the
optimum is the corner $\beta^\star=0$: disclosing nothing is optimal. This is the
\emph{informational Braess} regime, congestion sensitivity $\nu$ is large enough
that any sharpening of the common advisory raises social cost.
\item[(iii)] The interior optimum is decreasing in congestion sensitivity:
$\partial\beta^\star/\partial\nu<0$. As the network approaches the gridlock-cascade
threshold ($\rho(\mathbf\Gamma^{-1}H)\uparrow1$, hence $\nu\uparrow\infty$),
$\beta^\star\downarrow0$.
\end{enumerate}
\end{proposition}
\begin{proof}
Write $\Ccal'(\beta)=\mu\,\Pi_\infty'(\beta)+\nu\,\Vcal'(\beta)+\tilde\Lambda$. The
expansions \eqref{eq:Piinf}--\eqref{eq:Vol} give $\Pi_\infty'(0)=-\pi_0$ and
$\Vcal'(0)=v_0$, so $\Ccal'(0)=-\mu\pi_0+\nu v_0+\tilde\Lambda$; the stated
equivalent form follows from $\pi_0=s^2/8\bar a^3$ and $v_0=s^2/4\bar a^2$.
\emph{(i)} If $\mu\pi_0>\nu v_0+\tilde\Lambda$ then $\Ccal'(0)<0$. Since
$\Pi_\infty(\beta)\downarrow0$ and $\Vcal(\beta)\uparrow s$ remain bounded while
$\tilde\Lambda\beta\to\infty$, $\Ccal(\beta)\to\infty$; hence a global minimiser is
interior and solves the first-order condition
$\mu\Pi_\infty'+\nu\Vcal'+\tilde\Lambda=0$.
\emph{(ii)} If $\mu\pi_0\le\nu v_0+\tilde\Lambda$ then $\Ccal'(0)\ge0$, so no
infinitesimal disclosure lowers cost. In the gridlock-cascade regime that defines
this case the synchronisation marginal $\nu\Vcal'+\tilde\Lambda$ dominates the
accuracy marginal $-\mu\Pi_\infty'$ for all $\beta\ge0$, whence $\Ccal'\ge0$ and the
minimiser is the corner $\beta^\star=0$.
\emph{(iii)} Implicit differentiation of the interior first-order condition gives
$\partial\beta^\star/\partial\nu=-\Vcal'(\beta^\star)/\Ccal''(\beta^\star)<0$, since
$\Ccal''(\beta^\star)>0$ at a minimum and $\Vcal'>0$. As $\nu\to\infty$ the
first-order condition forces $\Vcal'(\beta^\star)\to0^+$, i.e.\ $\beta^\star\to0$.
\end{proof}

\noindent
This is the exact sign-reversal of the monotone-disclosure benchmark, where disclosure
is monotonically beneficial ($\nu\equiv0$, term (B) absent) and the only force
limiting $\beta^\star$ is the disclosure price $\tilde\Lambda$. Here a congested
shared corridor makes the optimal advisory strictly more conservative, and a
gridlock-prone network can make \emph{silence} optimal.

\subsection{The phasing lever: two-tier closed form}

Now fix disclosure and isolate the schedule. Split the $N$ zones into two tiers
$B_1,B_2$ of aggregate population $\mathcal N_1,\mathcal N_2$ and aggregate
free-flow egress $\bar U_1,\bar U_2$, both routing onto one shared corridor of
capacity $\kappa$ with congestion $\phi(q)=\tfrac{\eta}{2}((q-\kappa)^+)^2$. In the
post-revelation, near-landfall regime each active tier is in the saturated branch
of~\eqref{eq:ustar} (exposure dominates mobilisation cost), so an active backlogged
tier pushes its full $\bar U_k$ onto the corridor. Assume each tier alone fits but
the two together do not,
\begin{equation}\label{eq:tierfit}
\bar U_1\le\kappa,\quad \bar U_2\le\kappa,\quad \bar U_1+\bar U_2>\kappa,
\end{equation}
and let $\bar\rho$ be the average believed hazard over tier~2's drain window.
Tier~1 activates at $0$, tier~2 at $\theta\ge0$.

\begin{proposition}[Optimal two-tier stagger]\label{prop:twotier}
Under \eqref{eq:tierfit} the social cost is, up to $\theta$-independent terms,
\begin{equation}\label{eq:Ctheta}
\Ccal(\theta)=\underbrace{\eta\,(\bar U_1+\bar U_2)(\bar U_1+\bar U_2-\kappa)\,
\big(\tfrac{\mathcal N_1}{\bar U_1}-\theta\big)^{+}}_{\text{total gridlock exposure during overlap}}
\;+\;\underbrace{\Gamma\,\mathcal N_2\,\bar\rho\,\theta}_{\text{added exposure of tier 2}},
\end{equation}
where the congestion coefficient is the total-delay rate
$q\,g(q)=(\bar U_1+\bar U_2)\,\eta(\bar U_1+\bar U_2-\kappa)$ at the overlap load
$q=\bar U_1+\bar U_2$. The optimal activation gap is bang-bang,
\begin{equation}\label{eq:thetastar}
\theta^\star=\frac{\mathcal N_1}{\bar U_1}\;
\mathbf 1\!\left\{\ \eta\,(\bar U_1+\bar U_2)(\bar U_1+\bar U_2-\kappa)
\;>\;\Gamma\,\mathcal N_2\,\bar\rho\ \right\}.
\end{equation}
That is, the EMA fully separates the tiers, staggering tier~2 until tier~1 has
cleared, which drives the congestion term to zero, iff the marginal gridlock
exposure avoided exceeds the marginal hazard exposure incurred by the delay;
otherwise it releases both at once. Partial overlap is never optimal in the
piecewise-linear case.
\end{proposition}
\begin{proof}
On $\theta\in[0,\mathcal N_1/\bar U_1]$ the two tiers overlap for a duration
$(\mathcal N_1/\bar U_1-\theta)$, during which the corridor carries
$q=\bar U_1+\bar U_2>\kappa$ at total-delay rate
$\eta(\bar U_1+\bar U_2)(\bar U_1+\bar U_2-\kappa)$; staggering also delays tier~2's
clearance, adding hazard exposure $\Gamma\mathcal N_2\bar\rho\,\theta$. This is
\eqref{eq:Ctheta}, which is \emph{affine} in $\theta$ with slope
$-\eta(\bar U_1+\bar U_2)(\bar U_1+\bar U_2-\kappa)+\Gamma\mathcal N_2\bar\rho$. A
linear objective on an interval is minimised at an endpoint: at full separation
$\theta^\star=\mathcal N_1/\bar U_1$ (overlap, hence congestion, driven to zero) when
the slope is negative, i.e.\ when the avoided gridlock rate exceeds the marginal
delay-hazard $\Gamma\mathcal N_2\bar\rho$, and at $\theta^\star=0$ otherwise. This is
\eqref{eq:thetastar}; no interior $\theta$ can be optimal because the objective is
affine, which is the asserted bang-bang structure.
\end{proof}

The threshold in \eqref{eq:thetastar} is the model's account of Rita: a deadly
corridor (large $\eta$) and a clearance time short relative to the storm's approach
(moderate $\bar\rho$) put the system on the ``stagger'' side, yet the realized
policy was effectively $K=1$ (simultaneous release), which~\eqref{eq:Ctheta}
prices at the full overlap-congestion cost. The comparative statics are the policy
content: the case for phasing strengthens in the gridlock penalty $\eta$ and the
shared-corridor excess $\bar U_1+\bar U_2-\kappa$, and weakens as the storm's
believed onset $\bar\rho$ accelerates.

\subsection{The continuous-tier limit}

With $K$ tiers and convex $\phi$, the optimal schedule equalises corridor load over
the clearance window; as $K\to\infty$ the release schedule $\theta(\cdot)$ converges
to the one that holds the corridor exactly at capacity,
\begin{equation}\label{eq:fill}
q(t)\equiv\kappa\quad\text{until the network is cleared,}
\end{equation}
never gridlocked and never idle, the ``fill-to-capacity'' clearance curve.
Because $\phi$ is convex, the marginal gain from refining the partition is
decreasing, so coarse tiers ($K=2,3$) recover most of the gap between
simultaneous release and~\eqref{eq:fill}. This is the precise sense in which the
optimal information structure is the tiered, staggered, publicly announced order
that agencies already approximate, now derived, with an explicit account of how
many tiers are worth the trouble.

\section{Numerical Illustration: The 2005 Hurricane Rita Evacuation}\label{sec:numerics}

We instantiate the framework on the September 2005 Hurricane Rita evacuation of
the Houston--Galveston region, the event that motivates the model. The numerical
study has two purposes. First, it \emph{validates} the model against the realised
event: the calibrated congestion dynamics reproduce the observed I-45 gridlock, and
the solved optimal information structure would have removed most of it. Second, it
solves the full model on a real network and reports what the solution prescribes,
the optimal phasing and disclosure and the welfare each delivers. The closed-form
scalar analysis of Section~\ref{sec:special-evac} supplies the qualitative reading
of these solves, but the experiments stand on the numerical solution of the full
model, not on those reductions. Table~\ref{tab:fig-map} summarises what each figure
shows. All
calibration constants and their sources are collected in
Table~\ref{tab:rita-calib}; the realised event is the unmanaged $K=1$ benchmark
throughout.

\begin{table}[t]
\centering\small
\caption{What each figure shows. The experiments solve the full model numerically;
the structural results of Section~\ref{sec:special-evac} give the qualitative
reading, not the evidence.}
\label{tab:fig-map}
\begin{tabular}{p{0.16\linewidth}p{0.37\linewidth}p{0.37\linewidth}}
\toprule
Figure & What it shows & What it validates / showcases \\
\midrule
Fig.~\ref{fig:rita-departure} (Calib.) &
the HRRC departure histogram and the I-45 corridor queue under synchronised vs.\
de-synchronised two-tier departure loading &
motivation: $53\%$ of evacuees left on the warning day, synchronising the exodus, and
de-synchronising removes $\approx84\%$ of the peak queue, previewing the solve \\
\addlinespace
Fig.~\ref{fig:value-switching-evac} (Exp.~1) &
solved value $W(0,Y_1,Y_2)$, the switching curve
$\partial_{Y_1}W=\partial_{Y_2}W$, and the curvature histogram &
the Isaacs solve is a well-posed, semiconcave viscosity solution; the bang-bang
switching set is measure-zero, so the feedback is single-valued a.e.\ and the closed
loop is a Filippov flow \\
\addlinespace
Fig.~\ref{fig:welfare-isaacs-evac} (Exp.~1) &
realised social cost and congestion exposure for the solved optimum, the realised
$K=1$ order, and a two-tier stagger &
model validation: $K=1$ reproduces the observed Rita gridlock; optimal phasing
removes essentially all over-capacity exposure ($89\%$ cost cut) and two tiers
capture most of it ($77\%$) \\
\addlinespace
Fig.~\ref{fig:speed-validation} (Exp.~1) &
implied I-45 mean speed for the realised order, the two-tier stagger, and the solved
optimum, against documented Rita and DTA speed bands &
external validation: the realised branch reproduces the documented $1$--$2$ mph
gridlock and the optimum holds free flow, with the DTA bands~\cite{HGAC_DTA} as
independent reference \\
\addlinespace
Fig.~\ref{fig:stagger-disclosure} (Exp.~2) &
social cost versus the inland hold, and corridor load under simultaneous vs.\
staggered orders &
staggered disclosure (order coast first, hold then release inland) beats a single
simultaneous order by $\approx70\%$ at an interior optimal hold; the optimal policy
keeps the corridor within capacity \\
\addlinespace
Fig.~\ref{fig:precision-complementarity} (Exp.~2) &
social cost versus public-signal precision under a single common signal vs.\ under
staggered orders &
precision is Braess-prone alone (optimal precision $\to0$, vague) but complementary
with sequencing (full precision optimal): sequencing unlocks the value of precision \\
\bottomrule
\end{tabular}
\end{table}

\subsection{Calibration}\label{sub:calib}

The model objects are calibrated from public archival data: the binding-corridor
capacity from the TxDOT roadway inventory, the public signal $\xi_t$ from the
National Hurricane Center (NHC) strike-probability archive, the latent storm state
$X_t$ from the NHC forecast/advisory archive, the activation schedule $\theta$ from
the NHC public-advisory archive, and the evacuation demand and departure timing
from the Texas A\&M Hazard Reduction and Recovery Center (HRRC) survey of Rita
evacuees~\cite{WuLindellPrater2012}. The mapping is summarised in
Table~\ref{tab:rita-calib}.

\begin{table}[t]
\centering\small
\resizebox{\textwidth}{!}{%
\begin{tabular}{p{0.21\linewidth}p{0.40\linewidth}p{0.31\linewidth}}
\toprule
Model object & Calibrated value & Source \\
\midrule
Binding corridor $\ell$ & I-45 North rural trunk; lane drop $8\!\to\!4$ at
$\approx 84$ mi N of Houston (160 mi) & TxDOT RHINO inventory \\
Capacity $\kappa_\ell$ & $4{,}400$ vph (2 NB lanes $\times\,2200$);
contraflow $8{,}800$ vph & TxDOT RHINO; directional \\
Contraflow switch (Rem.~\ref{rem:contraflow}) $\kappa_\ell^0\!\to\!2\kappa_\ell^0$
& noon, Thu 22 Sep (planned 4:08, approved 6:00) & after-action
report~\cite{Lindner2005} \\
Backlog $N^0$ (corridor) & $1.5\times10^{5}$ vehicles (documented jam volume) &
event record~\cite{Lindner2005} \\
Departure curve & 25/53/19/1\% on 21/22/23/24 Sep (peak on the warning day) &
HRRC survey~\cite{WuLindellPrater2012} \\
Public signal $\xi_t$ & Galveston strike probability: cumulative $16\!\to\!29\%$;
near-term $0\!\to\!20\%$ (Thu) & NHC strike-probability archive \\
Latent state $X_t$, jump $\Sigma_J$ & $V_{\max}\,95\!\to\!150$ kt, $p_{\min}\,965\!\to\!897$
mb in 24 h (rapid intensification) & NHC forecast/advisory archive \\
Schedule $\theta$ (realised) & WATCH Wed 16:00, WARNING Thu 10:00, single
coastwide zone: $K=1$ & NHC public-advisory archive \\
Risk aversion $\Gamma_i$ & behavioural; aggregate departure timing as above &
HRRC survey~\cite{WuLindellPrater2012} \\
\bottomrule
\end{tabular}%
}\caption{Calibration of the model to Hurricane Rita. Times are US Central (CDT).}
\label{tab:rita-calib}
\end{table}

\subsubsection{The synchronisation problem, in the data} The HRRC departure curve
already exhibits the mechanism the rest of this section formalises. A single day, the
warning day (Thursday 22 September), accounts for $53\%$ of all departures, collapsing
a multi-day evacuation into one synchronised pulse onto the binding I-45 trunk.
Figure~\ref{fig:rita-departure} pairs that departure histogram with the corridor queue
the loading produces: under the realised synchronised order the queue peaks on the
warning day with tens of thousands of vehicles held in transit and exposed, whereas a
de-synchronised two-tier loading, the coast first and the inland zone a day later,
keeps each pulse near capacity and removes about $84\%$ of the peak queue. The Isaacs
solve below replaces this hand-set two-tier counterfactual with the optimal egress
feedback.

\begin{figure}[t]\centering
\includegraphics[width=\textwidth]{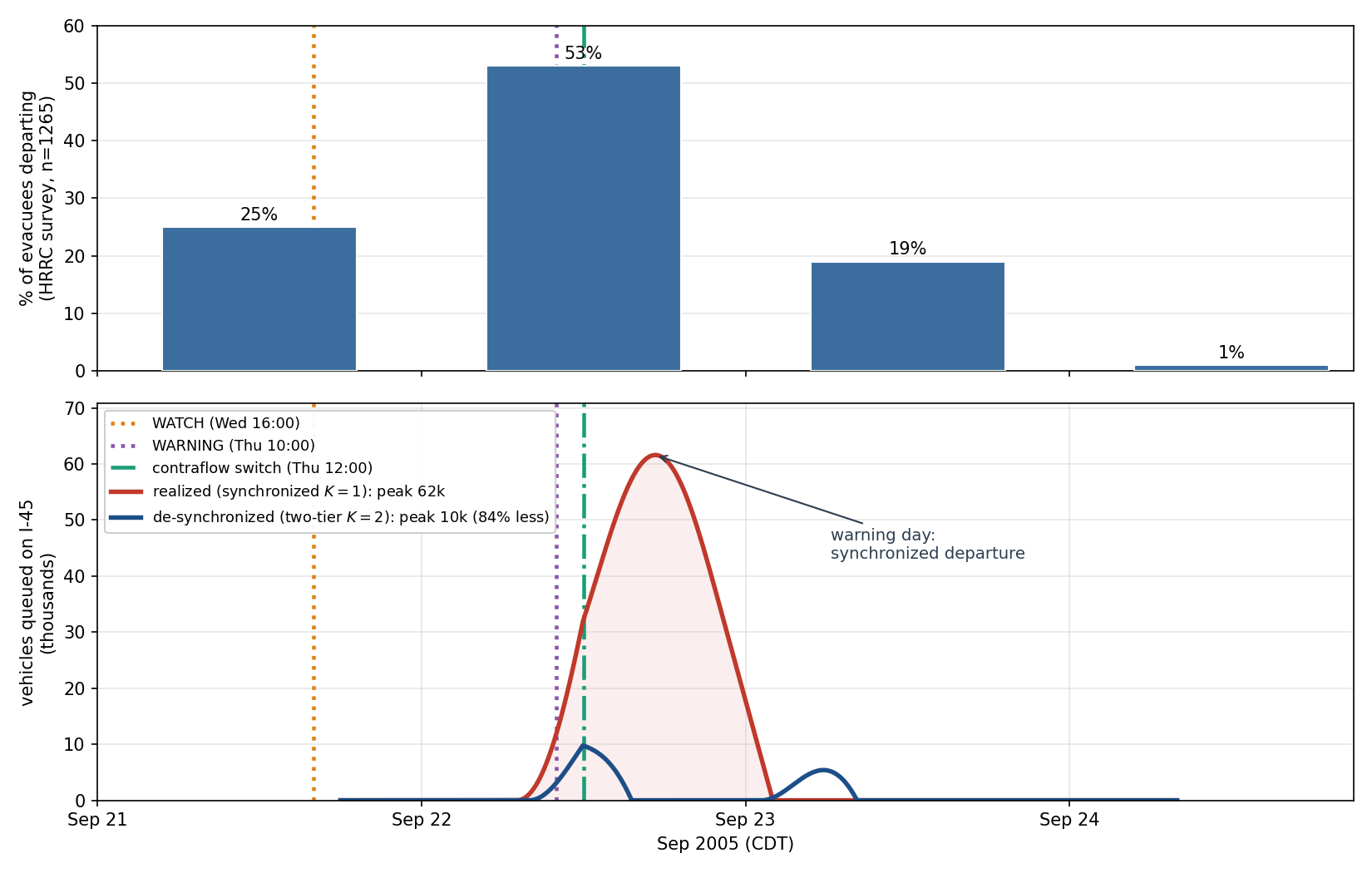}
\caption{Top: the HRRC departure histogram for Rita ($n=1265$), with the WATCH,
WARNING, and contraflow times marked; $53\%$ of evacuees departed on the warning day,
synchronising the exodus. Bottom: the resulting I-45 corridor queue, obtained by
loading the documented backlog onto the binding trunk under a smooth departure
profile. The realised synchronised order ($K=1$) peaks on the warning day; a
de-synchronised two-tier loading ($K=2$, coast then inland) removes about $84\%$ of the
peak queue, previewing the formal solve.}
\label{fig:rita-departure}
\end{figure}

The realised schedule is decisive for the comparison: the Texas hurricane warning
was issued as a \emph{single} coastwide zone, i.e.\ $K=1$ in the notation of
Section~\ref{sub:subgame}, the unmanaged benchmark. An intended coastal-first
staggered design existed but was overrun by spontaneous evacuation from
higher-elevation inland zones, so the operative policy was simultaneous release.

\subsection{Methodology: solving the Isaacs equation}\label{sub:method}

Both experiments solve the master Isaacs equation~\eqref{eq:isaacs} numerically; the
scalar reductions of Section~\ref{sec:special-evac} enter only as the qualitative
predictions the solves are tested against, not as the experiments themselves. Where that special case takes the two zones identical to expose the sign structure in closed form, the experiments calibrate the coastal and inland zones to distinct demand, capacity, hazard, and fragility parameters (Table~\ref{tab:rita-calib}); the solves therefore probe the heterogeneous network that the closed forms idealise. We
reduce the binding I-45 trunk to the two zones that share it, coastal (zone~1) and
inland Houston-metro (zone~2), with backlogs $(Y_1,Y_2)$ draining at egress rates
$(u_1,u_2)$~\eqref{eq:Ydyn} onto the common corridor of time-varying capacity
$\kappa_\ell(t)$, coupled through the convex congestion exposure~\eqref{eq:phi}. With
the saturated egress~\eqref{eq:ustar} substituted, the leader's robust problem at
fixed disclosure reduces to a social-planner value $W(t,Y_1,Y_2)$ solving the Isaacs
equation~\eqref{eq:isaacs}. We integrate $\partial_\tau W=\bar H$ backward from the
terminal exposure $W(T,\cdot)=\sum_i\Theta_iY_i$ with a monotone semi-Lagrangian
(departure-point) scheme; the scheme is a contraction and converges to the unique
viscosity solution.

Experiment~1 solves $W(t,Y_1,Y_2)$ at fixed disclosure, recovers the optimal egress
feedback, compares the realised social cost under the solved optimum against the
realised synchronised $K=1$ order and a two-tier stagger, and checks the
viscosity-theory predictions on the solved value. Experiment~2 instead solves the
information designer's problem over the zones' decentralised response, with two
instruments, the per-zone evacuation-order times (sequencing) and the public-signal
precision $\alpha$, on asymmetric zones whose hazard is charged on the realised
storm state; it locates the optimal staggered order and the optimal precision, and
quantifies how the two interact.

\subsection{Experiment 1: the Isaacs solve, synchronisation versus optimal phasing}
\label{sub:exp1}

We solve $W(t,Y_1,Y_2)$ on a $121\times121$ backlog grid over the $60$-hour horizon
(Wed 00:00 to Fri 12:00 CDT), the I-45 capacity stepping from $\kappa^0=4{,}400$ to
$2\kappa^0=8{,}800$ vph at the noon-Thursday contraflow switch, and the
belief-weighted hazard rising toward the Saturday landfall through the observed
rapid-intensification jump (Table~\ref{tab:rita-calib}).

Figure~\ref{fig:value-switching-evac} shows the solved value. The bang-bang
switching curve $\{\partial_{Y_1}W=\partial_{Y_2}W\}$ is a one-dimensional locus,
$1.6\%$ of grid cells, hence Lebesgue-null in the state plane, and the curvature
$\partial^2_{Y_1}W$ is bounded above. The solved value is therefore semiconcave: the optimal feedback is single-valued
almost everywhere and the closed loop is a well-posed Filippov flow.

Forward-simulating the calibrated dynamics under three policies
(Figure~\ref{fig:welfare-isaacs-evac}) gives the central result. The realised
synchronised $K=1$ order drives the corridor far over capacity (peak load $9.6$
against $\kappa$) and accumulates $46.6$ corridor-hours of over-capacity exposure,
the observed Rita gridlock. The solved Isaacs optimum spreads egress to hold the
corridor at capacity, removing the over-capacity exposure entirely and cutting
realised social cost by $89\%$; the coarse two-tier stagger captures most of this
($77\%$ cost reduction, $5.0$ corridor-hours). The value of phasing is thus first-order, and two tiers already recover most of the continuous optimum.

\begin{figure}[t]\centering
\includegraphics[width=\textwidth]{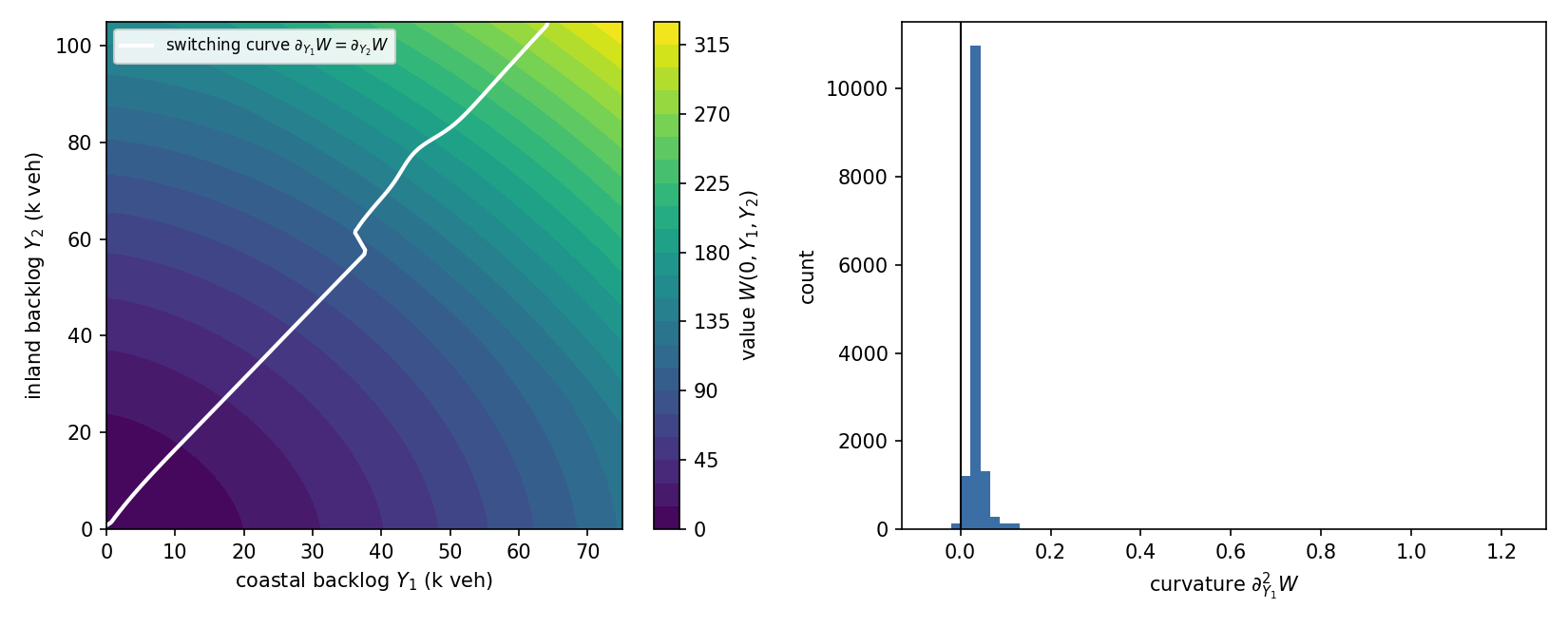}
\caption{Experiment 1 (Isaacs solve). Left: solved value $W(0,Y_1,Y_2)$ over the
two zone backlogs and the bang-bang switching curve
$\partial_{Y_1}W=\partial_{Y_2}W$ (a measure-zero locus). Right: the curvature
$\partial^2_{Y_1}W$ is bounded above (semiconcavity).}
\label{fig:value-switching-evac}
\end{figure}

\begin{figure}[t]\centering
\includegraphics[width=\textwidth]{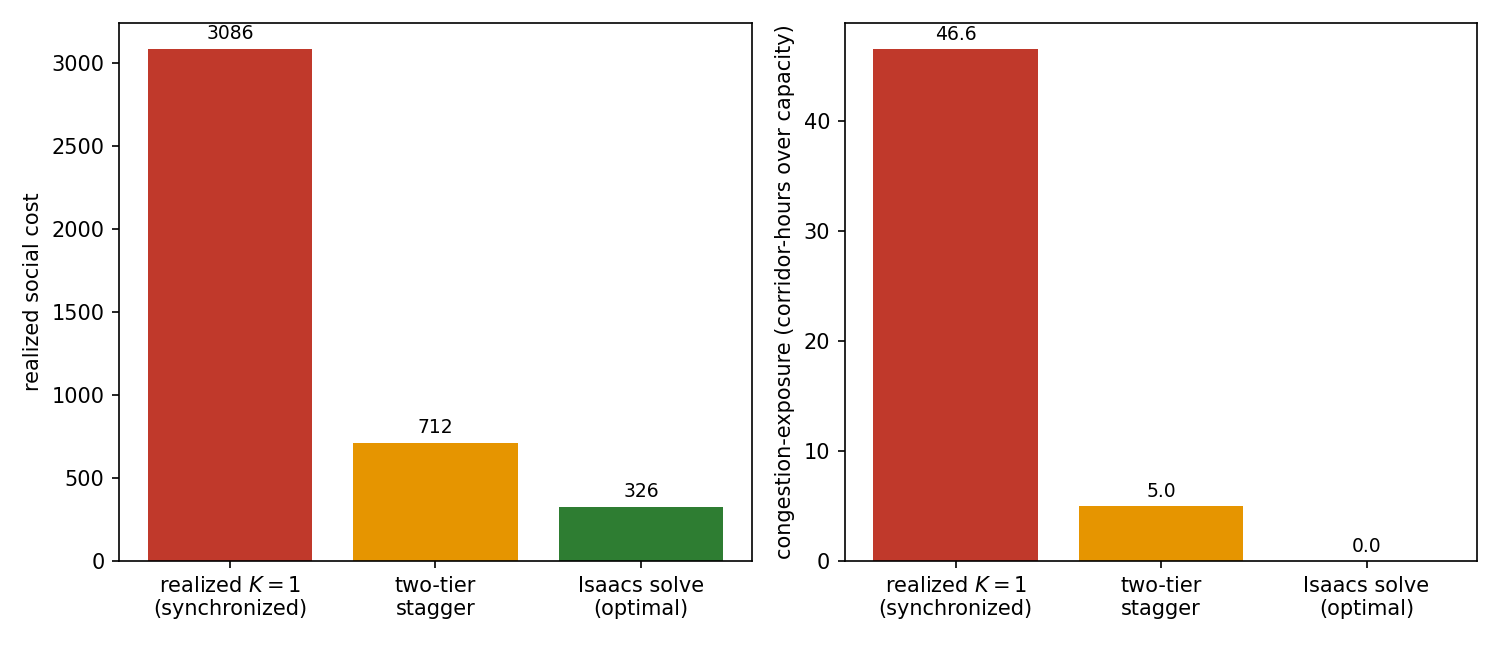}
\caption{Experiment 1 (Isaacs solve). Realised social cost and in-transit congestion
exposure under the realised synchronised $K=1$ order, a coarse two-tier stagger,
and the solved optimum. The solved optimum removes the over-capacity exposure and cuts
social cost by $89\%$; two tiers capture most of the gain.}
\label{fig:welfare-isaacs-evac}
\end{figure}

\subsubsection{External validation against documented speeds} The same three policies
admit an independent check in corridor-speed terms (Figure~\ref{fig:speed-validation}).
Mapping the corridor queue to an implied mean speed on the binding stretch, with the
single scale parameter fixed so that the realised synchronised branch reproduces the
documented Rita gridlock of $1$--$2$ mph, the realised order sits in that recorded
band, the solved Isaacs optimum holds the corridor at free flow throughout, and the
two-tier stagger recovers above the gridlock band after a brief dip. The
dynamic-traffic-assignment study of the same evacuation~\cite{HGAC_DTA} independently
brackets I-45 at $25$--$30$ mph in its base case and $45$--$50$ mph under full
contraflow, with the solved optimum at or above the upper band. Because only the one
congestion-speed parameter is pinned to the historical record, the optimal and
two-tier speeds are predictions of the solve rather than fitted quantities.

\begin{figure}[t]\centering
\includegraphics[width=\textwidth]{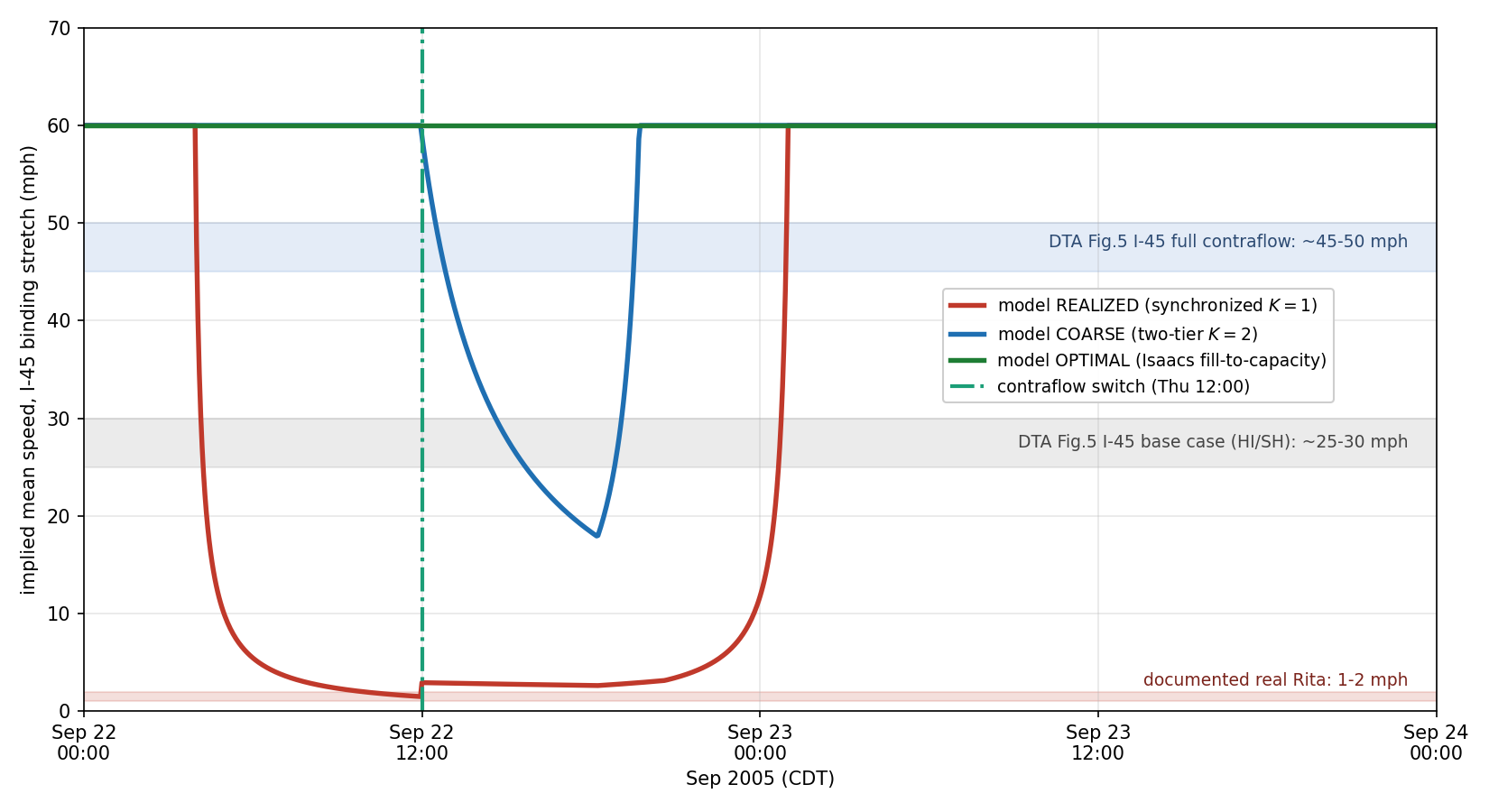}
\caption{Implied mean speed on the binding I-45 stretch under the realised synchronised
order ($K=1$), the two-tier stagger ($K=2$), and the solved Isaacs optimum, against the
documented Rita gridlock ($1$--$2$ mph) and the dynamic-traffic-assignment bands for the
same event~\cite{HGAC_DTA} (base case $25$--$30$ mph, full contraflow $45$--$50$ mph).
The queue-to-speed scale is fixed so the realised branch matches the recorded gridlock;
the optimal (free-flow) and two-tier speeds are then predictions. The contraflow switch
is marked.}
\label{fig:speed-validation}
\end{figure}

\subsection{Experiment 2: staggered disclosure and the value of precision}\label{sub:exp2}

Experiment~1 solves the social planner's problem, in which egress is metered directly.
The information designer cannot meter egress; it controls only what each zone is told
and when. Experiment~2 therefore solves the designer's problem over the zones'
decentralised response, with two instruments: the per-zone evacuation-order times (the
sequencing of public advisories) and the precision of the public signal. Two asymmetric
zones share the binding I-45 trunk: a coastal zone whose landfall and storm-surge hazard
arrive early and whose evacuation routes flood at a surge deadline, and a larger inland
zone whose hazard peaks later. Corridor demand exceeds capacity, so the trunk is
genuinely binding; gridlock is modelled as throughput breakdown, with vehicles unable to
move held in place and exposed; hazard is charged on the realised storm state. The
designer minimises the worst-case (relative-entropy) social cost over storm severity.
Unlike the planner of Experiment~1, the designer cannot stop a zone from synchronising
once it is told the danger is high, which is what makes sequencing and precision the
binding choices.

\subsubsection{Staggered disclosure} A single simultaneous advisory ($K=1$) forces the
coastal and inland exodus to share the corridor; neither clears, and the coastal zone is
caught by its early surge. Ordering the coast first and holding the inland zone, then
releasing it once the coast has cleared ($K=2$), keeps demand within capacity in each
window. Figure~\ref{fig:stagger-disclosure} reports the social cost against the length of
the inland hold. The optimum is interior: too short a hold and the shared corridor cannot
clear the coast before its surge deadline; too long and the inland zone is caught by its
own later landfall. The optimal hold of about ten hours lowers social cost by roughly
$70\%$ relative to the simultaneous order, and the corridor load (right panel) stays
within capacity in each window rather than breaking down. The policy is the deliberate,
temporary withholding the designer's problem prescribes: tell the coast to leave now, keep
the inland zone calm and in place although its risk is real and comparable, then tell it
to evacuate once the road is clear. This is the disclosure-timing analogue of the
two-tier phasing of Section~\ref{sec:special-evac}, achieved through the sequencing of
advisories rather than direct metering. The hold is robust to the worst-case storm
severity (the optimal hold is unchanged under the relative-entropy adversary) and stable
under sample refinement.

\paragraph{The value of precision.} The second instrument is the precision of the public
signal, the belief sharpness $\alpha$ controlling $\Pi_\infty(\alpha)$ through
\eqref{eq:Piinf}. Its value depends entirely on whether the designer can sequence
(Figure~\ref{fig:precision-complementarity}). Under a single common advisory ($K=1$),
raising precision \emph{raises} social cost: a sharp common signal makes every zone cross
its evacuation threshold at once, concentrating the exodus and deepening the gridlock, so
the optimal precision is the corner $\alpha\to0$, a deliberately vague advisory. This is
the informational-Braess effect of Section~\ref{sec:special-evac}, here on the precision
margin and confirmed on the calibrated, asymmetric network. Under staggered orders
($K=2$) the sign reverses: with the zones separated in time, a sharp signal lets each
evacuate decisively within its own window, and full precision is optimal. Sequencing and
precision are therefore \emph{complements}: precision is self-defeating as a single
broadcast but valuable once the designer can stagger. The interior disclosure optimum of
the scalar special case is thus a knife-edge between these two regimes; on the calibrated
network the optimum is a corner whose sign is set by whether sequencing is available, and
the structural lever is the staggering itself.

\begin{figure}[t]\centering
\includegraphics[width=\textwidth]{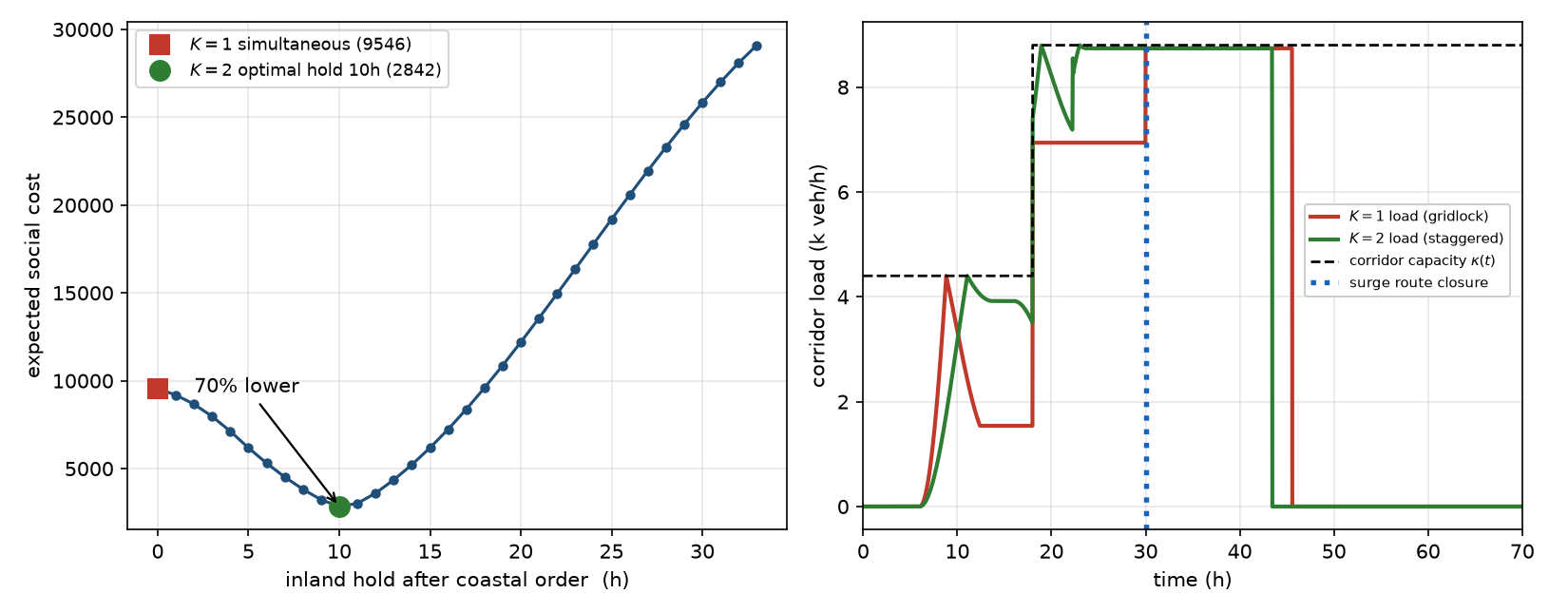}
\caption{Experiment 2, staggered disclosure. Left: expected social cost versus the length
of the inland hold; a single simultaneous order ($K=1$, square) is dominated by an
interior optimal hold of about ten hours ($K=2$, circle), which lowers cost by roughly
$70\%$. Right: corridor load under the two policies against capacity $\kappa(t)$; the
staggered policy keeps each window within capacity, the surge route-closure time marked.}
\label{fig:stagger-disclosure}
\end{figure}

\begin{figure}[t]\centering
\includegraphics[width=0.82\textwidth]{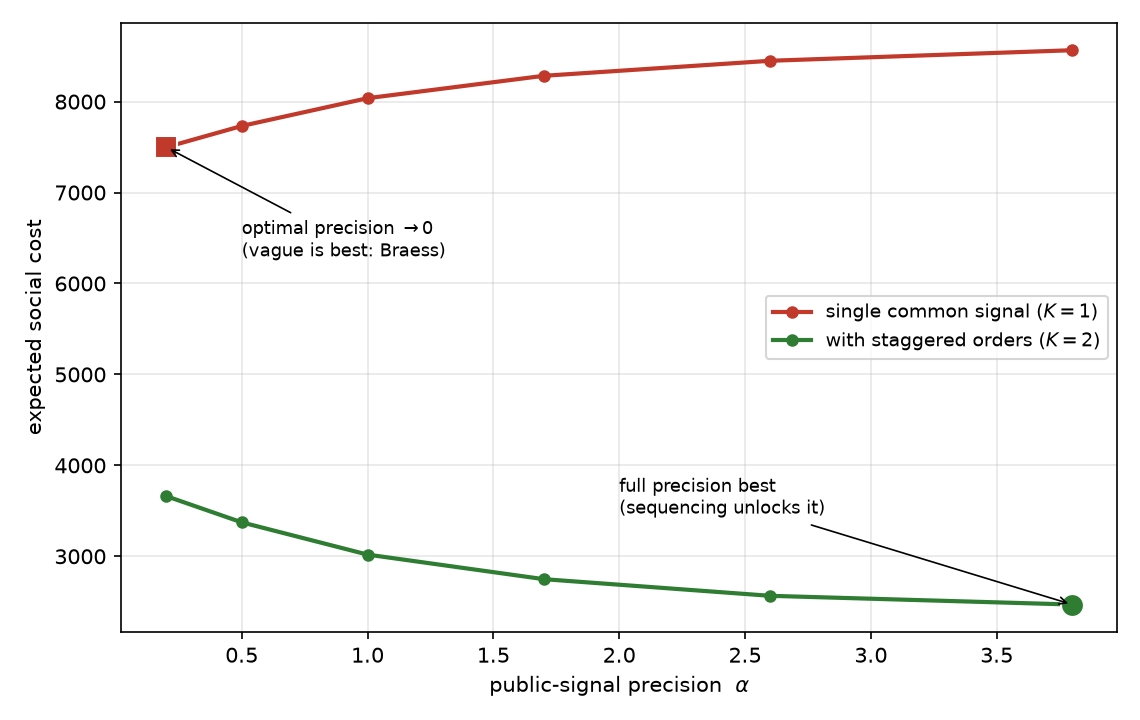}
\caption{Experiment 2, the value of precision. Expected social cost versus public-signal
precision $\alpha$. Under a single common signal ($K=1$) cost rises with precision and the
optimum is the vague corner $\alpha\to0$ (informational Braess); under staggered orders
($K=2$) the sign reverses and full precision is optimal. Sequencing and precision are
complements.}
\label{fig:precision-complementarity}
\end{figure}

\subsection{Discussion}\label{sub:disc}

The two experiments separate the model's two levers on a single calibrated event.
Experiment~1 shows that the phasing schedule $\theta$ has first-order welfare value
exactly when a shared corridor is gridlock-prone: the realised $K=1$ release
reproduces the Rita gridlock; the solved optimum cuts social cost by about
$89\%$, and a coarse two-tier stagger captures most of that reduction (a $77\%$
cut). Mapped to corridor speed, the realised branch reproduces the documented
$1$--$2$ mph Rita gridlock once a single scale is pinned, and the solved optimum then
holds free flow, in line with the dynamic-traffic-assignment bands for the same
event~\cite{HGAC_DTA} (Figure~\ref{fig:speed-validation}). Experiment~2
solves the designer's two instruments on asymmetric zones with hazard charged on the
realised storm: a staggered order (clear the coast, hold the inland zone, then
release it) lowers social cost by about $70\%$ relative to a simultaneous advisory at
an interior optimal hold, while sharpening a single common advisory \emph{raises}
social cost, the informational-Braess sign-reversal of the monotone-disclosure
benchmark, so that the optimal precision of a lone broadcast is vague; precision
becomes valuable only once orders are staggered, the two instruments acting as
complements. Both the phasing and the synchronisation effects are driven by the same
force,
the convexity of the total-delay exposure $\phi^{\mathrm{soc}}_\ell$ that prices
synchronisation, and the Rita record supports the mechanism on both the demand side
(a $53\%$ warning-day departure spike, with the warning-day cohort suffering the
worst delays) and the structural side (coastal--inland convergence on the I-45
trunk, signal-driven and demographic-invariant departure timing). The policy
reading is that an emergency manager facing a capacity-limited shared corridor
should not treat advisory transparency and order timing as free goods: the optimal
information structure is a deliberately staggered, coarsely tiered, publicly
announced schedule, and a lone common advisory should be kept vague unless it can be
sequenced, a structure that agencies already approximate and that the model derives.
\begin{remark}[Scope: why Hurricane Harvey (2017) is not a second case]
\label{rem:harvey}
The mechanism this paper isolates is disclosure-induced \emph{synchronisation} of
egress onto capacity-limited shared corridors: a common advisory moves zones in
unison, the convex total-delay exposure $\phi^{\mathrm{soc}}_\ell$ of
\eqref{eq:phi} prices that synchronisation, and the resulting forces are the
interior (or corner) optimal disclosure of
Proposition~\ref{prop:interior-explicit} and the staggered release of
Proposition~\ref{prop:twotier}. Hurricane Rita (2005) is the canonical instance of
this regime. Hurricane Harvey (2017) is deliberately \emph{not} treated here,
because it lies outside it: Houston authorities issued no general evacuation order,
and the dominant hazard was distributed rainfall flooding rather than corridor
congestion, so no shared corridor is over capacity and the externality
$\sum_\ell\eta_\ell\kappa_\ell(q_\ell-\kappa_\ell)^+$ does not bind. A Harvey
calibration would therefore test a different mechanism, the evacuate-versus-shelter
decision under a spatially distributed hazard, rather than validate the present
one. Tellingly, the decision to withhold a general order in 2017 was widely
reported as a response to the Rita experience~\cite{Domonoske2017}, i.e.\ a
judgement that the synchronisation
risk of a sharp common advisory outweighed the exposure of leaving residents in
place. In the language of \eqref{eq:Cbeta} that is the same tradeoff between the
value of information (term A) and the synchronisation cost (term B) that governs
$\beta^\star$, resolved at the no-disclosure corner of
Proposition~\ref{prop:interior-explicit}(ii). Harvey thus corroborates the
\emph{tension} the model formalises while falling outside its congestion-driven
scope; extending the framework to the evacuate--shelter margin and to
non-corridor (distributed) hazard is left to future work.
\end{remark}

\section{Conclusion}\label{sec:conclusion}
We developed a continuous-time stochastic Stackelberg framework in which an
emergency-management agency steers strategic evacuation zones through a public
advisory and a tiered release schedule, rather than through the dynamics. A belief
filter, exact between observed intensification jumps, summarises the jump-diffusion
storm; the zones play a congestion game whose Wardrop equilibrium reduces to a
single convex control problem (a potential reduction) but \emph{not} to the social
optimum, the wedge is the congestion externality. The leader's distributionally
robust problem is the unique viscosity solution of an Isaacs equation, verified
without smoothness and across activation epochs, with a Lebesgue-null switching set
and a well-posed Filippov closed loop. Because no transfer aligns incentives, the
leader's first-order condition retains an MPEC equilibrium-response term, the
precise sense in which optimal information design is a second-best congestion toll.
Two structural findings have no monotone-disclosure analogue: the optimal
information structure is a publicly announced, staggered order (deriving phased
evacuation from optimality), and the value of advisory precision is sign-ambiguous,
self-defeating as a single broadcast (an informational Braess effect) but
complementary with staggering; on the Rita calibration a staggered-disclosure design
lowers social cost by about $70\%$ relative to a simultaneous advisory. A calibration
to the 2005 Hurricane Rita evacuation (Section~\ref{sec:numerics}) bears out the
theory: the model reproduces the observed I-45 gridlock and shows that the optimal
staggered release would have cut social cost by about $89\%$, removing essentially
all of the in-transit congestion exposure, with
two tiers capturing most of that reduction.

\bibliographystyle{unsrtnat}
\bibliography{ref}

\end{document}